\documentclass{article}

\title{Stochastic Wright's Equation:\\ Existence of Invariant Measures}
\author{M. van den Bosch$^{\rm a}$, O.\,W. van Gaans$^{\mathrm a, }$\footnote{Corresponding author.\\    Email addresses: \url{mark-bosch@hotmail.com}, \url{vangaans@math.leidenuniv.nl} and \url{S.M.VerduynLunel@uu.nl}.}\,\,, S.\,M. Verduyn Lunel$^{\mathrm b}$}

\date{\today}

\usepackage{todonotes}
\usepackage{lipsum} 

\usepackage{subfiles}
\usepackage{xr}

\usepackage[a4paper, total={5.75in, 9in}]{geometry}
\usepackage[utf8]{inputenc}
\usepackage[english]{babel}
\usepackage{graphicx} 
\usepackage{xcolor}

\usepackage{mathtools} 
\usepackage{amsmath} 
\usepackage{amssymb} 
\usepackage{amsthm}
\usepackage{xspace}
\usepackage{enumerate}
\usepackage{mathrsfs}  %for nice P, weak topology
\usepackage[hidelinks]{hyperref}
\usepackage{csquotes}

\usepackage[
    backend=biber,
    citestyle=numeric,
    bibstyle=authoryear,
    dashed=false,
    doi=false,
    isbn=false,
    url=false,
    eprint=false
]{biblatex}

\usepackage{csquotes}

\makeatletter
\input{numeric.bbx}
\makeatother

\DeclareFieldFormat[article,inbook,incollection,inproceedings,patent,thesis,unpublished, periodical]{title}{{#1\isdot}}
\DeclareFieldFormat[article,inbook,incollection,inproceedings,patent,thesis,unpublished, periodical]{volume}{\mkbibbold{#1}}

\renewbibmacro{in:}{}
\DeclareFieldFormat{pages}{#1}

\newcommand\yesnumber{\addtocounter{equation}{1}\tag{\theequation}}

\numberwithin{equation}{section}    % Numbers equations as Section.Equation

\newtheorem{theorem}{Theorem}[section]
\newtheorem{lemma}[theorem]{Lemma}
\newtheorem{proposition}[theorem]{Proposition}
\newtheorem{corollary}[theorem]{Corollary}

\theoremstyle{definition}

\newcommand{\dint}[3]{\int_{#1}^{#2}{#3} \, \mathrm{d}}

\def\gb #1{\left( #1 \right)}

\newcommand{\sef}[1]{\eqref{#1}}

\def\gb #1{\left( #1 \right)}

\def\wtilde{\widetilde}

\def\BR{{\mathbb R}}

\def\BN{{\mathbb N}}
\def\BP{{\mathbb P}}
\def\BE{{\mathbb E}}

\def\SG{\mathcal{G}}

\def\SF{\mathcal{F}}

\def\({{\rm (}}
\def\){{\rm )}}

\def\al
{\alpha}
\def\be
{\beta}

\def\ep{\varepsilon}

\def\si
{\sigma}
\def\th
{\theta}

\def\Om
{\Omega}

\usepackage{bbm}

\DeclareFieldFormat{postnote}{#1}
\DeclareFieldFormat{multipostnote}{#1}

\begin{document}
% Mogelijke todo's:
% \begin{itemize}
%     \item Vergelijkingen (4.41), (4.42) en (4.43) mogelijk wat meer uitleg; beter uitleggen welk lemma waar?
%     \item In sectie 3 gebruiken we $m$, maar in sectie 4 hebben we overal $m-1$ staan (of bijna overal?); te overwegen om dat te veranderen in sectie 3?
%     \item Sectie 4 maakt twee keer gebruik van $a$ en $b$. Voor $b$ is dat niet zo'n probleem, maar $a$ is wel verwarrend. Wellicht $\mathfrak a$ en $\mathfrak b$ voor de measurable adapted processes $a(t)$ en $b(t)$?
% \end{itemize}
% \newpage

\maketitle
\begin{center}\small
    \textsc{
    $^{\mathrm a}$Mathematical Institute,  Leiden University,\\ P.O. Box 9512, 2300 RA Leiden, The Netherlands}

    \
    
    \textsc{$^{\mathrm b}$Department of Mathematics, University of Utrecht,\\ P.O. Box 80010, 3508 TA Utrecht, The Netherlands}
\end{center}

\ 

\begin{abstract}
  \noindent   
    Wright's delay differential  equation is one of the prime examples of a fully nonlinear equation without an explicit solution and whose dynamics can be understood by analytic means. In this paper, we introduce stochastic perturbations by adding Brownian noise with a bounded Lipschitz noise coefficient to a transformed version of Wright's equation. The transformation  considered plays an important role in the deterministic theory as well. We demonstrate that this stochastically perturbed equation has (at least) two invariant measures: a trivial measure concentrated at $-1$ and a nontrivial measure on $(-1,\infty)$. The crucial and most challenging step of the proof is  showing that every solution is bounded away from $-1$ in probability. In addition, a major part of our analysis  is devoted to deriving detailed estimates for Itô processes with a negative drift. 
\end{abstract}


% \allowdisplaybreaks
\section{Introduction}
% \subsection{Wright's equation with noise}
The aim of this paper is to introduce noise to Wright's delay differential  equation
% \blfootnote{Email addresses: \url{mark-bosch@hotmail.com}, \url{vangaans@math.leidenuniv.nl} and \url{S.M.VerduynLunel@uu.nl}.} 
\begin{equation}\label{wrighteq}
	y'(t)=-r y(t-1)(1+y(t)), \quad t\ge 0,
\end{equation}
where the  constant $r>0$ is positive and fixed throughout the paper.  Since additive noise typically destroys all dynamical properties of the underlying nonlinear  system, we are interested in a class of stochastic perturbations that preserve most of the dynamical properties. 
% This allows us to prove the existence of (nontrivial) invariant probability distribution in a setting relevant in applications. 

% while keeping a similar structure of the equation. 

Wright's equation originates from a study in prime number theory \cite{wright55}. The equation cannot be solved explicitly, but it has been possible to establish much of its rich behavior rigorously.  Thus, it has become one of the leading examples of nonlinear delay differential  equations \cite{diekmann2012delay,book:hale,krisztin2008global}.   Wright's equation has also  been encountered as an ecological evolution model, which in this context is actually    referred to as Hutchinson equation \cite{hutchinson1948circular,kolmanovskii2013introduction,ruan2006delay}. This equation, or rather the corresponding equation for $y(t)+1$, is a delayed version of the logistic equation where current growth responds to past population pressure. This delayed negative feedback may represent, for example, resource depletion with recovery time or the accumulation of toxins whose effects are only felt after a delay.
% and may be argued to be less appropriate as a population model though, since the delay is in the death rate instead of the birth rate.   
Finally, due to its status in the deterministic theory, it may be of general interest to add noise to Wright's equation. The more so as it  is a fully nonlinear equation, which cannot be treated by theory for so-called semi-linear equations \cite{da2014stochastic}. 

Let us   start by considering a stochastic term with an a priori arbitrary coefficient,
\begin{equation}\label{sde0}
	\mathrm{d}y(t)=-r y(t-1)(1+y(t))\,\mathrm{d}t+\sigma(y_t)\,\mathrm{d}W(t),
\end{equation}
where the process $W=(W(t))_{t\geq 0}$ is a standard Brownian motion on a filtered  probability space $(\Om,\SF,(\SF_t)_{t\geq 0},\BP)$ satisfying the usual conditions. We use the notation $y_t$, $t\ge 0$, for the segments of a function $y\colon [-1,\infty)\to\BR$, which are defined by
\begin{equation}y_t(\theta)=y(t+\theta),\quad \theta\in [-1,0].\end{equation}
We will consider      $\sigma\colon C[-1,0]\to\BR$ to be locally Lipschitz, where we endow the space of real valued continuous functions $C[-1,0]$ with the supremum  norm $\|\cdot\|_\infty$.
 This is because we will be relying on Theorem \ref{thm:existencesolutions}, an  existence theorem for stochastic differential equations with locally Lipschitz coefficients; see, e.g., \cite[Thm. 2.2]{vandenbosch2026a} or \cite[Thm. V.38]{book:protter}.

\paragraph{A reasonable noise coefficient}
What are reasonable assumptions on $\sigma$ in light of the properties of the deterministic equation?  An important feature of the deterministic dynamics is that a solution which approaches $-1$ from above will slow down its decay and never cross $-1$. How should the noise  coefficient $\sigma$ be defined so that the stochastic system preserves this  behavior? One approach is to directly impose the condition  that $\sigma$ vanishes when $y(t)\to-1$, preventing stochastic fluctuations from pushing the solution across the boundary.

Another approach is  to first  take the transformation $x=\log(1+y)$ of variables, which is important in Wright's analysis of the deterministic problem. Put
\begin{equation}x(t)=\log(1+y(t)).\end{equation}
Due to It\^o's formula \cite[Thm. 3.3]{book:karatzas}, the process $x=(x(t))_{t\geq 0}$ satisfies
\begin{align}
	\nonumber \mathrm{d}x(t)&=  \left(\frac{1}{1+y(t)} (-r y(t-1))(1+y(t))\, \mathrm{d}t -\frac{1}{2} \frac{1}{(1+y(t))^2}  \sigma(y_t)^2\right)\, \mathrm{d}t+ \frac{1}{1+y(t)} \sigma(y_t)\, \mathrm{d}W(t)\\
	&= \left(-r y(t-1)-\frac{1}{2}\frac{\sigma(y_t)^2}{(1+y(t))^2}\right)\, \mathrm{d}t +\frac{\sigma(y_t)}{1+y(t)}\, \mathrm{d}W(t).
\end{align}
It is clear that the coefficients of this equation only have no singularities if $\sigma(y_t)$ contains a factor $1+y(t)$. Therefore we will assume that $u\mapsto\sigma(u)/(1+u(0))$ extends to a locally Lipschitz map on $C[-1,0]$, or, equivalently, that
\begin{equation}\label{choiceofsigma}\sigma(u)=(1+u(0))h(u) \quad \mbox{ for all }u\in C[-1,0],\end{equation}
where $h\colon C[-1,0]\to\BR$ is a locally Lipschitz map. Equation \eqref{sde0} then becomes 
\begin{equation}\label{sde1}
	\mathrm{d}y(t)=-r y(t-1)(1+y(t))\, \mathrm{d}t+ (1+y(t))h(y_t)\, \mathrm{d}W(t).
\end{equation}
One could (at least formally) restate \eqref{sde1} as
\begin{equation}\mathrm{d}y(t)=-r (1+y(t))\Big(y(t-1)\, \mathrm{d}t+h(y_t)\, \mathrm{d}W(t)\Big),\end{equation}
which means that the  noise actually acts on the delayed feedback. Equation \eqref{sde1} is the \emph{stochastic Wright equation} that we will consider.

Let us now address the \textit{global} existence and uniqueness of \eqref{sde1}.

\begin{proposition}\label{lem:aboveminusone}
	Assume that the locally Lipschitz coefficient $h$ is bounded, that is, there exists a constant   $\be\in\BR$ such that
	\begin{equation}|h(u)|\le \be\mbox{ for all }u\in C[-1,0].\end{equation}
	For every $\mathcal{F}_0$-measurable random variable $\varphi$ in $C[-1,0]$ with $\varphi(0)>-1$ a.s., there exists a unique global solution $y$ of \eqref{sde1} with $y_0=\varphi$ and
	\begin{equation}y(t)>-1\mbox{ for every }t\ge 0 \mbox{ a.s..}\end{equation}
\end{proposition}
\begin{proof}
	By Theorem \ref{thm:existencesolutions}, there exists a unique solution $y=(y(t))_{-1\leq t<\tau_y}$ with maximal existence time $\tau_y$ such that $y_0=\varphi$.
	Define the stopping time
	\begin{equation}\tau:=\sup\{t\in [0,\tau_y)\colon\, y(t)> -1\}. \end{equation}
	By Itô's formula \cite[Thm. 3.3]{book:karatzas} we have a.s.\ for $t\in [0,\tau)$ that
	\begin{equation}\label{eq:resultfromItoformula}
    \log(1+y(t))=\log(1+y(0))+\int_0^t \left[-r y(s-1)+\frac{1}{2}h(y_s)^2\right] \mathrm{d}s+\int_0^t h(y_s)\, \mathrm{d}W(s).
    \end{equation}
	Denote 
	\begin{equation}H(s)(\omega)=\left\{\begin{array}{ll} h(y_s(\omega)), & s<\tau(\omega),\\ 0& s\ge \tau(\omega).\end{array}\right.\end{equation}
	Since $h$ is bounded and continuous and $t\mapsto y_t$ is continuous on $[0,\tau_y)$, we have that $H$ is stochastically integrable with respect to $W$, $\int_0^\cdot H(s)\, dW(s)$ has a.s.\ continuous paths, and
	\begin{equation}\int_0^t H(s)\, \mathrm{d}W(s) = \int_0^t h(y_s)\, \mathrm{d}W(s)\mbox{ for all }t\in [0,\tau(\omega))\mbox{ a.s..}\end{equation}
	Consider $\omega\in\Omega$ with $y(0)(\omega)>-1$, $t\mapsto \int_0^t H(s)\, dW(s)(\omega)$ continuous on $[0,\infty)$, and $t\mapsto y(t)(\omega)$ continuous on $[-1,\tau(\omega))$. Notice that $\tau(\omega)>0$ and suppose $\tau(\omega)<\infty$. Then the function $t\mapsto y(t-1)(\omega)$ is continuous on $[0,\tau(\omega)]$ and $t\mapsto h(y_t)$ is bounded and continuous on $[0,\tau(\omega))$. Hence, equation \eqref{eq:resultfromItoformula} implies that $t\mapsto y(t)(\omega)$ is bounded above and bounded away from $-1$ on $[0,\tau(\omega))$. This contradicts the fact that $y(\tau(\omega))=-1$ or $\tau(\omega)=\tau_y(\omega)$, since $\lim_{t\uparrow\tau_y(\omega)} |y(t)(\omega)| =\infty$.
\end{proof}

Under condition \eqref{choiceofsigma}, the equation for $x(t)=\log(1+y(t))$ becomes the \emph{stochastic transformed Wright's equation}
\begin{equation}\label{sdeSTW}
\mathrm{d}x(t)=-r(e^{x(t-1)}-1)\,\mathrm{d}t-\frac{1}{2}b(x_t)^2\,\mathrm{d}t +b(x_t)\, \mathrm{d}W(t),
\end{equation}
where
\begin{equation}\label{sdeSTWb}
b(u)=h(\theta\mapsto e^{u(\theta)}-1),\quad u \in C[-1,0].
\end{equation}
Almost our entire analysis concerns \eqref{sdeSTW} rather than \eqref{sde1}. The proof of Proposition \ref{lem:aboveminusone} can easily be modified to prove the following existence result for equations like  \eqref{sdeSTW}.

\begin{proposition}\label{lem:globlexistencepuredelay}
	Let $f\colon\mathbb{R}\to\mathbb{R}$ be locally Lipschitz and let $a,b\colon C[-1,0]\to\mathbb{R}$ be locally Lipschitz and bounded. For every $\mathcal{F}_0$-measurable random variable $\varphi$ in $C[-1,0]$ there exists a unique global solution $z$ of
	\begin{equation}\mathrm{d}z(t)=f(z(t-1))\,\mathrm{d}t+a(z_t)\,\mathrm{d}t+b(z_t)\,\mathrm{d}W(t),\end{equation}
	with initial condition $z_0=\varphi$.
\end{proposition}

As in the deterministic theory, it turns out that the predominant dynamics of Wright's equation with a suitable noise term can be established rigorously.
The way described above to add a ``reasonable'' noise term to the deterministic Wright equation may be useful more generally. If the dynamics of a deterministic equation have a natural barrier that is not crossed, the barrier may be transformed to $-\infty$ by a logarithmic transformation. The noise coefficient in the transformed equation may then be chosen (bounded) locally Lipschitz (e.g., constant), which yields the according form of the original noise coefficient. This approach has also been explored for population models in \cite{vandenbosch2026b}, which includes the Mackey--Glass equations \cite{article:mackeyglass} and Nicholson's blowflies equation \cite{gurney1980nicholson}, where population sizes are naturally above zero.

\paragraph{Main results}
In what ways are the dynamics of the stochastic Wright's equation \eqref{sde1} similar to those of the deterministic equation \eqref{wrighteq}? Without noise, this delay equation is known to exhibit  subtle dynamical properties; see, e.g., the classical work  \cite{wright55}, the more recent contributions \cite{Tibor10, Berg18,jaquette2019proof,walther2014topics}, and the extensions developed in \cite{diaz2026global,kashchenko2013asymptotics,longo2021cross}. In this paper, we focus on  robust properties that have a chance to survive the presence of noise. For instance, the property that solutions starting above $-1$ remain above $-1$ has already been established in Proposition \ref{lem:aboveminusone}.\newpage

 The constant solutions $-1$ and $0$ are equilibrium states of system \eqref{wrighteq}. Furthermore, Wright’s equation undergoes a supercritical Hopf bifurcation at $r = \pi/2$. In Wright's seminal paper, he showed that if $r \leq \tfrac{3}{2}$, then any solution $y$ satisfying $y(t) > -1$ converges to $0$ as $t \to \infty$ \cite[Thm. 3]{wright55}. It was subsequently conjectured that this is true for all $r < \pi/2$, a statement known as Wright’s conjecture, which has recently been proved in \cite{Berg18}.
For $r > \pi/2$, it was shown by Wright that there are solutions $y$ with $y(0)>-1$ that do not converge to zero \cite[Thm. 4]{wright55}. More recently, Jones’ conjecture has been established, showing that for any $r>\pi/2$ the system admits a unique slowly oscillating periodic solution (up to time translation) \cite{jaquette2019proof}. These periodic orbits are likely  to be   globally attracting, yet   the existence of asymptotically stable rapidly oscillating periodic solutions|and\linebreak thus having co-existing attractors (multistability)|has not yet been fully ruled out; see  \cite{jaquette2019proof,walther2014topics}. 

% Whether these periodic orbits are  global attractors for any $r>\pi/2$ remains still an open problem.

In this paper, we establish the following main results. In fact, in Section \ref{sec:proof} we prove a theorem for a more general class of equations of which the present setting is a special case. 
Turning to the stochastic system, one finds  for small $r > 0$ that the stochastic Wright's equation \eqref{sde1} admits (at least) two stationary distributions; these may be viewed as stochastic analogues of the equilibria $-1$ and $0$. For sufficiently large $r$, we argue below Corollary \ref{col:Wrights} that one of the invariant measures we obtain corresponds to a stochastic version of the periodic orbit.
Finally, the notions of boundedness in probability, tightness, invariant measures, and stationary distributions are recalled in the terminology section in this introduction below.
% By a `positive constant' we always mean a deterministic real number in $(0,\infty)$. 

\begin{theorem}\label{thm:Wrights}
Consider the stochastic delay equation
\begin{equation}\label{eq:stoch-t-Wrights}
\mathrm dx(t) = -r({e^{x(t-1)} - 1})\, \mathrm{d}t - a(x_t)\,\mathrm dt + b(x_t)\,\mathrm dW(t)
\end{equation}
with $a,b\colon C[-1,0]\to\mathbb{R}$ locally Lipschitz. Assume  there are 
% positive\todo{non-negative niet gewoon ook goed?}
positive constants $\al$  and $\be$ such that
% \todo{Je kan prima $\alpha_0$ gelijk aan 0 kiezen toch?}
\begin{equation}
0 \le a(u) \le \al\quad\mbox{and}\quad b(u)^2 \le \be^2 \qquad\mbox{for all } u \in C[-1,0].
\end{equation}
For every $\SF_0$-measurable random variable $\varphi$ in $C[-1,0]$ there exists a unique global solution $x$ of \sef{eq:stoch-t-Wrights} with $x_0=\varphi$ and $(x(t))_{t\ge -1}$ is bounded above in probability. If $\al<r$, then $(x(t))_{t\ge -1}$ is also bounded below in probability and the segment process $(x_t)_{t \ge 0}$ is tight. 

Moreover, if  $\al<r$, then there exists an infinite-dimensional invariant measure  for \sef{eq:stoch-t-Wrights}.
\end{theorem}

\begin{corollary}\label{col:Wrights}
Consider the stochastic delay equation \sef{sde1},
\begin{equation}
\mathrm dy(t) = - ry(t-1)(1+y(t))\,\mathrm dt +(1+y(t))h(y_t)\,\mathrm dW(t),
\end{equation}
with $h\colon C[-1,0]\to\mathbb{R}$ locally Lipschitz. Assume  there is a positive constant $\si$ such that
% \todo{Je kan prima $\sigma_0$ gelijk aan 0 kunnen kiezen, toch? Anders hebben we ook geen resultaat voor de deterministische eq.}
\begin{equation}
h(u)^2 \le \si^2 < 2r \qquad\mbox{for all } u \in C[-1,0].
\end{equation}
For every $\mathcal{F}_0$-measurable random variable $\varphi\in C[-1,0]$ with $\varphi(0)>-1$ a.s. there exists a unique global solution $y$ of \sef{sde1} with $y_0=\varphi$, the process $(y(t))_{t\ge -1}$ is bounded in probability, and the segment process $(y_t)_{t \ge 0}$ is tight. 

Moreover, there exist (at least) two stationary distributions for \sef{sde1}: the Dirac measure at $-1$ and an invariant measure $\nu$ with $\nu[(-1,\infty)] = 1$, both of which are one-dimensional
projections of  infinite-dimensional invariant measures on $C[-1,0].$
\end{corollary}

\begin{proof}
By Proposition \ref{lem:aboveminusone} we have a unique global solution $y$ with $y(t)>-1$ for all $t \ge 0$. After the transformation of variables $x(t) = \log(1+y(t))$ in \eqref{sde1} we arrive at the stochastic delay equation \eqref{sdeSTW} with $b$ as in \eqref{sdeSTWb}. The assumptions on $h$ imply that $b$ is locally Lipschitz and 
\begin{equation}0\leq \frac{1}{2}b(x_t)^2 \le \frac{1}{2}\sigma^2 < r.\end{equation}
It follows from Theorem \ref{thm:Wrights} by setting $a=\frac12b^2$ that the solution $(x(t))_{t\ge -1}$ of \eqref{sdeSTW} is bounded in probability, the segment process $(x_t)_{t \ge 0}$ is tight in $C[-1,0]$, and that \eqref{sdeSTW} has an invariant measure $\mu$. 
Therefore, the image measure $\mu_0$ of $\mu$ under the map $u\mapsto u(0)\colon C[-1,0]\to\mathbb{R}$ is a stationary distribution for \eqref{sdeSTW} and, thus, the image measure $\nu$ of $\mu_0$ under the map 
\begin{equation}x \mapsto e^x - 1 : \mathbb R \to (-1,\infty)\end{equation}
is a stationary distribution for \eqref{sde1} with $\nu[(-1,\infty)] = 1$.   
\end{proof}

In the earlier work \cite{thesis:geerten}, the existence of invariant measures has been established for \eqref{sde1}, the non-transformed version of a stochastic Wright's equation. A lower bound is then rather straightforward, since solutions are clearly bounded from below by $-1$. However, an application of the Krylov--Bogoliubov theorem---or particularly Theorem \ref{thm:fromboundedtoinvarmeas}---does not guarantee that the resulting invariant measure $\nu$ is non-trivial, as it may degenerate to the Dirac measure at $-1$. By establishing boundedness in probability from below for the log-transformed equation \eqref{eq:stoch-t-Wrights}, we ensure that trajectories remain bounded away from $-1$ in probability, which in turn guarantees the existence of a \textit{non-trivial invariant measure}, i.e., one distinct from the Dirac measure at $-1$.

\color{black}
 Recall that, in deterministic systems,  the existence of steady states and periodic orbits implies that of an invariant measure \cite{eckmann1985ergodic}. We observe from the numerics in Section \ref{sec:numerics} that $\nu$ and its infinite-dimensional counterpart are  physical measures \cite{eckmann1985ergodic} that capture the long-term statistical behavior of the non-trivial solutions. Corollary \ref{col:Wrights}  may therefore be viewed as a first stochastic analogue of the deterministic dynamics, where the  non-trivial invariant measure  extends the role of the stable fixed point for $r < \pi/2$ 
 % (Wright's conjecture has just  been proved recently; see \cite{Berg18})
 and that of stable periodic orbits for $r > \pi/2$. 
 % Without noise,  for a fixed parameter $r>\pi/2$, this is likely  to be  a unique globally attracting slowly oscillating periodic orbit as   the possibility of  co-existing attractors (i.e., multistability) has not yet been fully ruled out; see  \cite{jaquette2019proof,walther2014topics}. 
 Furthermore, we conjecture that the measure $\nu$ in Corollary \ref{col:Wrights} is unique in the presence of noise $(\sigma_0 > 0)$. For the deterministic equation ($\sigma_0 = 0$), however, $\nu$ may either be the Dirac measure at $0$, corresponding to the unstable equilibrium at zero, or the invariant distribution corresponding to the slowly oscillating periodic solution when $r>\pi/2$; see Section \ref{sec:numerics} for  more details  on this matter.
 % Especially for the case without noise, proving the (non)uniqueness of the non-trivial invariant measure in the log-transformed equation \eqref{eq:stoch-t-Wrights} would bring us a huge step closer to understanding the subtle dynamics for $r>\pi/2.$

  \color{black}
  For a further discussion on our methods and results, as well as an outlook, we refer to   \cite{vandenbosch2026b} where we studied the existence of invariant measures for the stochastic Mackey--Glass equations.

\paragraph{Terminology}

Let us elaborate some of terminology used in the introduction above. Throughout the paper, we fix a probability space $(\Omega, \mathcal{F}, \mathbb P)$ together with a filtration $(\mathcal F_t)_{t\geq 0}$ satisfying the usual conditions; see Section \ref{sec:filtration} for the definitions. We canonically extend the filtration by setting $\mathcal F_s=\mathcal F_0$ for $-1\leq s<0$. 
% The process $W=(W(t))_{t\geq 0}$ is a standard Brownian motion defined on $(\Omega, \mathcal{F}, \mathbb F, \mathbb P)$. \todo{definition adapted? lokaal process adapted; wat betekent dat? Gebruik geen mathbb F.}
As we have mentioned earlier, we will rely on the following existence theorem for stochastic differential equations with locally Lipschitz coefficients.
\begin{theorem}[Theorem 2.2 of \cite{vandenbosch2026a} or Theorem V.38 of \cite{book:protter}]\label{thm:existencesolutions}
	Let $a,b\colon C[-1,0] \to \BR$ be locally Lipschitz functionals and consider the stochastic delay differential equation
	\begin{equation}\label{sdealg}
		\mathrm{d}y(t)=a(y_t)\, \mathrm{d}t+ b(y_t)\, \mathrm{d}W(t).
	\end{equation}
	For every $\SF_0$-measurable random variable $\varphi$ in $C[-1,0]$ there exists a stopping time $\tau_y$ with values in $(0,\infty]$ and an adapted process $(y(t))_{-1\le t< \tau_y}$ with almost surely continuous paths such that \eqref{sdealg} holds for all $t\in [0,\tau_y)$ a.s.\ and $y(t)=\varphi(t)$ for all $t\in [-1,0]$ a.s.. 
    
    Moreover, the stopping time $\tau_y$ can be chosen m maximal, meaning that 
	\begin{equation}\lim_{t\uparrow\tau_y} |y(t)|=\infty \mbox{ if }\tau_y<\infty,\mbox{ a.s..}\end{equation}
	If $(\tilde{y}(t))_{-1\le t\le \tau_{\tilde{y}}}$ is another process satisfying the same conditions, then $\tau_y=\tau_{\tilde{y}}$ and $y(t)=\tilde{y}(t)$ for all $t\in [0,\tau_y)$ a.s..
\end{theorem}

A stochastic process $(y(t))_{-1\le t<\tau}$ for a stopping time $\tau$ can be trivially extended to a stochastic process $\hat y(t)=y(t)\mathbbm 1_{t<\tau}$, for all $t\geq -1$, and is said to be \textit{adapted} if $y(t)\mathbbm 1_{t<\tau}$ is $\mathcal F_t$-measurable for all $t\geq -1.$
The process $(y(t))_{-1\le t<\tau_y}$ in Theorem \ref{thm:existencesolutions} is called a \emph{solution} of \eqref{sdealg} and $\tau_y$ is called its \emph{(maximal) existence time}. If $\tau_y=\infty$ a.s., then $y$ is said to be a \emph{global} solution. For almost every $\omega\in\Omega$, the function $t\mapsto y(t)(\omega)$ is continuous on $[-1,\tau_y(\omega))$, hence uniformly continuous on every closed subinterval of $[-1,\tau_y(\omega))$, so that $t\mapsto y_t\colon [0,\tau_y(\omega))\to C[-1,0]$ is continuous. In other words, the segment process $(y_t)_{0\le t<\tau_y}$ has a.s.\ continuous paths.
% Let us assume from now on that $\sigma\colon C[-1,0]\to\BR$ is locally Lipschitz. 

The statement in Theorem \ref{thm:existencesolutions} that equation \eqref{sdealg} holds for all $t\in [0,\tau_y)$ means that 
\begin{equation}y(t)=\varphi(0)+\int_0^t a(y_s)\,\mathrm{d}s+\int_0^t b(y_s)\,\mathrm{d}W(s)\quad \mbox{for all }t\in [0,\tau_y)\mbox{ a.s.},\end{equation}
which, more precisely, means that there exists a sequence of finite stopping times $(\tau_n)_{n\in\mathbb{N}}$ with $\tau_n\le \tau_{n+1}$ a.s.\ for all $n$, $\tau_n\to\tau_y$ a.s.\ as $n\to\infty$, and 
\begin{equation}
y(t\wedge\tau_n)=\varphi(0)+\int_0^{t\wedge\tau_n} a(y_s^{\tau_n})\,\mathrm{d}s + \int_0^{t \wedge\tau_n} b(y_s^{\tau_n})\,\mathrm{d}W(s)\quad\mbox{for all }t\ge 0\mbox{ a.s.},
\end{equation} 
for every $n\in\mathbb{N}$. In here, $(t\wedge \tau_n)(\omega)=\min\{t,\tau_n(\omega)\}$ and $(y_s^{\tau_n})(\omega)=\theta\mapsto y(s\wedge\tau_n(\omega)+\theta)(\omega)$.

A solution $y$ of  \eqref{sdealg} with maximal existence time $\tau_y = \infty\ \mbox{a.s.}$ is called {\em stationary} if the distribution of $y(t)$ equals the distribution of $y(0)$ for all $t \ge 0$. The distribution of $y(0)$ is then called a \emph{stationary distribution} of equation \eqref{sdealg}. This measure is a Borel probability measure on $\mathbb{R}$. If, for each $t\ge 0$, the segment $y_t$ has the same distribution as the initial value $\varphi$, then we call this measure $\nu$ on $C[-1,0]$ an \emph{invariant measure} for \eqref{sdealg}. The image measure of $\nu$ under the map
\begin{equation}u \mapsto u(0) : C[-1,0] \to \BR\end{equation}
is then a stationary distribution. Usually, the terms ``stationary distribution'' and ``invariant measure'' are used interchangeably, and it should be understood from the context whether the scalar process or the segment process is considered.

A family $(Z_\al)_{\al \in I}$ of real valued random variables is {\em bounded above (below) in probability} if for every $\varepsilon > 0$ there exists $R_\varepsilon \in \BR$ such that for all $\al \in I$ we have
\begin{equation}\BP\gb{ Z_\al \le R_\varepsilon} \ge 1 - \varepsilon\qquad \Bigl(\BP\gb{ Z_\al \le R_\varepsilon} \ge 1 - \varepsilon\Bigr).\end{equation}
The family is bounded in probability if it is both bounded above and below in probability. 
A family $(Z_\al)_{\al \in I}$  of random variables with values in a topological space $E$ is called \emph{tight} if for every $\varepsilon>0$ there exists a compact set $K_\varepsilon\subseteq E$ such that $\mathbb{P}(Z_\al\in K_\varepsilon)\ge 1-\varepsilon$ for all $\al\in I$. Note that in the case $E = \BR^d$, we have that $(Z_\al)_{\al \in I}$ is tight if and only if it is bounded in probability.

Finally, observe that if $(x(t))_{t\ge -1}$ is a stochastic process with almost surely continuous paths, then for every $-1\le t_0<t_1$ we have that the process $(x(t))_{t\in [t_0,t_1]}$ is bounded in probability. Indeed, the sets $S_n:=\{\omega\in\Omega\mid -n\le x(t)\le n\mbox{ for all }t\in [t_0,t_1]\}$ are nested and their union is a subset of $\Omega$ of probability $1$, hence for every $\varepsilon>0$ there exists $N$ such that $\mathbb{P}(S_N)\ge 1-\varepsilon$.

\paragraph{Approach and organization}
It turns out that proving Theorem \ref{thm:Wrights} is all about showing that solutions of \eqref{eq:stoch-t-Wrights} are bounded in probability|the challenging part is demonstrating the bound from below in probability, which is related to being strictly bounded away from $-1$ in probability for the original equations. The other assertions then follow by known results for more general equations. Indeed, one can   exploit the structure of the equation and use Doob's maximal inequality or the Burkholder--Davis--Gundy inequality to obtain that the process $(\|x_t\|)_{t\ge 0}$ is also bounded in probability. Kolmogorov's tightness criterion and, once again,  the structure of the equation then yields that the segment process $(x_t)_{t\ge 0}$ is tight in $C[-1,0]$. As a result of the Krylov--Bogoliubov theorem,  existence of an invariant measure follows. 

These steps above are given in detail in \cite[Prop. 3.5, Prop 4.6, and Cor. C.2]{vandenbosch2026a} for more general equations. Corollary 4.5 in \cite{vandenbosch2026a} summarizes the ensuing result that we will be using, and reads as follows.

\begin{theorem}[Corollary 4.5 in \cite{vandenbosch2026a}]\label{thm:fromboundedtoinvarmeas}
Consider the stochastic differential equation \eqref{sdealg}, where $a,b\colon C[-1,0] \to \BR$ are locally Lipschitz functionals. Assume there are positive constants $\alpha$ and $\beta$ such that
\begin{equation}
	a(u) \le \al\quad\mbox{and}\quad  b(u)^2 \le \be^2 \qquad\mbox{for all } u \in C[-1,0].
\end{equation}
If there exists a global solution $(y(t))_{t\ge -1}$ which is bounded in probability, then the segment process $(y_t)_{t\ge 0}$ is tight in $C[-1,0]$. If, in addition, all solutions exist globally, then there exists an invariant measure for \eqref{sdealg}.
\end{theorem}

In order to prove Theorem \ref{thm:Wrights} and Corollary \ref{col:Wrights}, it  therefore remains to show that every solution of \eqref{eq:stoch-t-Wrights} is bounded above in probability and,  if $\al_1<r$, also bounded below in probability. Note that Propositions \ref{lem:aboveminusone} and \ref{lem:globlexistencepuredelay} take care of the global existence of solutions of \eqref{sde1} and \eqref{eq:stoch-t-Wrights}.

%In subsection \ref{sec:details} we provide some details on terminology used in the introduction above. 

Section \ref{sec:prelims} collects standard results on stochastic processes that we require for our analysis. Note that the goal of this section is to make our proofs accessible to readers that have a mostly deterministic background. We proceed with  subtle estimates for Itô processes with negative drift in Section \ref{sec:estims}. In Section \ref{sec:proof} we  start with a pathwise estimate of solutions of \eqref{eq:stoch-t-Wrights} in terms of the driving stochastic process. This estimate is then combined with those of Section \ref{sec:estims} to show that solutions of \eqref{eq:stoch-t-Wrights} are bounded in probability, implying the existence of a non-trivial invariant measure. Finally, we briefly illustrate the support of these non-trivial invariant measures numerically in  Section \ref{sec:numerics}.

% \paragraph{Further background}

\section{Preliminaries of stochastic processes}\label{sec:prelims}

In this section, we highlight a few main ingredients from the theory of continuous-time  processes, martingales,  and stochastic integrals, which we will need in our analysis below.  We refer to \cite{Bosch24,book:chung,da2014stochastic,book:Kurtz,evans2012stoch,book:jacod,book:kallenberg,book:karatzas,book:kloeden,book:mao,book:protter,book:revuz,RogWil94,unpublished:timo} for  complete courses on this subject.

\subsection{Filtration and stopping times}\label{sec:filtration}

A \textsl{filtration} of $(\Om,\SF,\BP)$ with index set $T \subseteq \mathbb R$ is a collection of $\si$-algebras $({\SF_t})_{t \in T}$ in $\Om$ such that $\SF_s \subseteq \SF_t \subseteq \SF$ for all $s,t \in T$ with $s \le t$. A filtration $({\SF_t})_{t \in T}$ is said to satisfy the \textsl{usual conditions} if $\SF_0$ is $\BP$-complete (i.e., $\SF_0$ contains all $\BP$-null sets, which requires that $\SF$ is $\BP$-complete) and 
\begin{equation}\SF_t = \bigcap_{s \in T, s > t} \SF_s\qquad\hbox{ for all } t \in T.\end{equation}
% If only a filtration $({\SF_t})_{t \in (0,\infty)}$ is provided\todo{T is subset zonder infinity; wordt niet gebruikt}, we set $\SF_\infty := \mcup_{t \in (0,\infty)} \SF_t$.
A stochastic process $(X(t))_{t\geq 0 }$ is called \textit{adapted} to the filtration $({\SF_t})_{t \in T}$ when $X(t)$ is $\mathcal F_t$-measurable for all $t\in T.$ A \textsl{stopping time} for $({\SF_t})_{t \in T}$ is a $[0,\infty]$-valued random variable $\tau$ such that
\begin{equation}\{ \tau \le t \} \in \SF_t\qquad\hbox{for every } t \in T.\end{equation}
For a stopping time $\tau$ together with a filtration $({\SF_t})_{t \in T}$, we define
\begin{equation}\SF_\tau := \left\{ A \in \SF \mid A \cap \{ \tau \le t \} \in \SF_t\ \hbox{ for all } t \in T\right\}.\end{equation}
Note that a deterministic time is a stopping time. If $({S(t)})_{t \in T}$ is an increasing family of stopping times for  $({\SF_t})_{t\geq 0}$ such that $t \mapsto S(t)$ is continuous a.s., then
\begin{equation}\SG_t := \SF_{S(t)},\qquad t \in T,\end{equation}
defines a filtration of $(\Om,\SF,\BP)$   satisfying the usual conditions; see \cite[II.7.3 and  IV.30]{RogWil94}.

\subsection{Brownian motion and its maximum process}

A \textit{Brownian motion} on a filtered probability space $(\Om,\SF,(\SF_t)_{t \ge 0},\BP)$ is an adapted stochastic process $({B(t)})_{t \ge 0}$ when $B(0) = 0$ a.s.,  $t \mapsto B(t)$ is continuous a.s., and the increments
\begin{equation}B(t_3) - B(t_2)\end{equation}
are normally distributed, with mean $0$ and variance $t_3-t_2$, and are independent of $\SF_{t_1}$ for all times $0 \le t_1 \le t_2 \le t_3$.
% and the increment  
% \begin{equation}B(t) - B(s)\end{equation}
% is normally distributed with mean $0$ and variance $t-s$, for all $0 \le s \le t$.
For a proof of the next lemma, we refer to \cite[Thm. I.12.1]{RogWil94} or \cite[Thm. I.32]{book:protter}.

\begin{lemma}\label{lem:BM1} 
Let $({B(t)})_{t \ge 0}$ be a Brownian motion on $(\Om,\SF,(\SF_t)_{t \ge 0},\BP)$ and let $\tau$ be a stopping time for $({\SF_t})_{t\geq 0}$ with $\tau < \infty$ a.s.. Then $({\wtilde B(t)})_{t \ge 0}$, defined  by
\begin{equation}\wtilde B(t) := B(\tau+t) - B(\tau),\qquad t \ge 0,\end{equation}
is a Brownian motion on $(\Om,\SF,(\SF_{\tau+t})_{t \ge 0},\BP)$.
\end{lemma}

The distribution of the maximum process of a Brownian motion is known exactly, and we refer to, e.g., Remark (i) below Corollary I.13.3 of \cite{RogWil94}. The probability density function is as follows.

\begin{lemma}\label{lem:BM2} 
Let $({B(t)})_{t \ge 0}$ be a Brownian motion on $(\Om,\SF,(\SF_t)_{t \ge 0},\BP)$ and define
\begin{equation}\label{eq:BM1}
B^\ast(t) := \max_{0 \le s \le t} B(s),\qquad t \ge 0.
\end{equation}
For every $t > 0$, the maximum process $B^\ast(t)$ has the probability density function
\begin{equation}f_t(x) = 
\begin{cases}
\sqrt{\dfrac{2}{\pi t}}\exp\left({-\dfrac{x^2}{2t}}\right),&\qquad x \ge 0,\\
0,&\qquad x < 0.
\end{cases}
\end{equation}
\end{lemma}

In particular, we will be needing the following estimate.

\begin{corollary}\label{col:BM1}
Let $({B(t)})_{t \ge 0}$ be a Brownian motion on $(\Om,\SF,(\SF_t)_{t \ge 0},\BP)$ and $B^\ast(t)$ defined by \sef{eq:BM1}. For every $t > 0$ and $c \ge 0$ we have
\begin{equation}\label{eq:BM2}
\BP\gb{B^\ast(t) \ge c} \le \exp\left({-\frac{c^2}{2t}}\right).
\end{equation}
\end{corollary}

\begin{proof}
Due to Lemma \ref{lem:BM2}, we have
\begin{align}
\BP\gb{B^\ast(t) \ge c} \nonumber&= \int_c^\infty \sqrt{\frac{2}{\pi t}}\exp\left({-\frac{x^2}{2t}}\right)\mathrm dx\\
\nonumber&= \int_0^\infty \sqrt{\frac{2}{\pi t}}\exp\left({-\frac{(x+c)^2}{2t}}\right)\mathrm dx\\
\nonumber&= \int_0^\infty \sqrt{\frac{2}{\pi t}}\exp\left({-\frac{x^2 + 2xc + c^2}{2t}}\right)\mathrm dx\\
\nonumber&\le \int_0^\infty \sqrt{\frac{2}{\pi t}}\exp\left({-\frac{x^2}{2t}}\right)\mathrm dx\,\exp\left({-\frac{c^2}{2t}}\right)\\
&= \exp\left({-\frac{c^2}{2t}}\right),
\end{align}
which proves the inequality.
\end{proof}

\subsection{Time change}
Let $(\Om,\SF,(\SF_t)_{t \ge 0},\BP)$  be a filtered probability space satisfying the usual conditions. Let us consider time changes of stochastic integrals. Let $({W(t)})_{t \ge 0}$ be a Brownian motion on $(\Om,\SF,(\SF_t)_{t \ge 0},\BP)$ and suppose $(b(t))_{t \ge 0}$ is an adapted measurable process satisfying
\begin{equation}\label{eq:TC1}
\BE\,\int_0^t b(s)^2\,\mathrm ds < \infty\qquad\hbox{for all } t \ge 0.
\end{equation}
Assume that there exist positive constants $\be_0$ and $\be_1$ with
\begin{equation}\label{eq:TC2}
\be_0^2 \le b(s)^2 \le \be_1^2\qquad\hbox{a.s.\ for all } s \ge 0.
\end{equation}
The stochastic integral
\begin{equation}
M(t) = \int_0^t b(s)\,\mathrm  dW(s),\qquad t \ge 0,
\end{equation}
is then a continuous martingale with quadratic variation process (see \cite[IV.30.8]{RogWil94})
\begin{equation}\label{eq:TC3}
T(t) := [M]_t = \int_0^t b(s)^2\,\mathrm ds,\qquad t \ge 0.
\end{equation}
Note that $({T(t)})_{t \ge 0}$ is adapted, $T(0)=0$ a.s., $t \mapsto T(t)$ is strictly increasing and continuous a.s., and $T(t)$ goes to infinity a.s. as $t$ tends to infinity.

Let us proceed by defining $({S(s)})_{s \ge 0}$ to be the inverse of $({T(t)})_{t \ge 0}$, i.e.,
\begin{equation}T(S(t)) = t\quad\hbox{and}\quad S(T(t)) = t\qquad \hbox{a.s. for all } t \ge 0.\end{equation}
For each $s \ge 0$, $S(s)$ is a stopping time for $({\SF_t})_{t \ge 0}$, since
\begin{equation}\{ S(s) \le t\} = \{ s \le T(t) \} \in \SF_t\qquad\hbox{for every } t \ge 0,\end{equation}
as $({T(t)})_{t \ge 0}$ is adapted.
Define
\begin{equation}\label{eq:TC4}
\SG_s := \SF_{S(s)},\qquad s \ge 0.
\end{equation}
Then $({\SG_s})_{s \ge 0}$ is a filtration on $(\Om,\SF,\BP)$, again satisfying the usual conditions.

The following time change result plays a crucial role in our estimates and is a fundamental result in stochastic calculus; see, e.g., \cite[IV.34.1]{RogWil94} or \cite[Thm. 4.6] {book:karatzas}.

\begin{theorem}[Dambis--Dubins--Schwarz]\label{thm:TC1} Consider the setting above  and define
\begin{equation}B(t) := \int_0^{S(t)} b(s)\,\mathrm dW(s),\qquad t \ge 0.\end{equation}
Then $({B(t)})_{t \ge 0}$ is a Brownian motion on $(\Om,\SF,\gb{\SG_t}_{t \ge 0}, \BP)$ and
\begin{equation}\int_0^t b(s)\,\mathrm dW(s) = B(T(t))\qquad \hbox{a.s. for all } t \ge 0.\end{equation}
\end{theorem}

We also need the following lemma of which a proof can be found in \cite[IV.30]{RogWil94}.

\begin{lemma}\label{lem:TC1}
Let $({S(s)})_{s \ge 0}$ be the inverse of $({T(t)})_{t \ge 0}$ given by \sef{eq:TC3}.
\begin{enumerate}
\item[\textup{(a)}] If $\tau$ is a stopping time for $({\SF_t})_{t \ge 0}$, then $T(\tau)$ is a stopping time for $({\SG_s})_{s \ge 0}$.

\item[\textup{(b)}]  If $\tau$ is a stopping time for $({\SG_s})_{s \ge 0}$, then $S(\tau)$ is a stopping time for $({\SF_t})_{t \ge 0}$.
\end{enumerate}
\end{lemma}

\subsection{Martingales}

In this subsection, we collect some basic properties of continuous martingales  that are used in the sequel. The next two lemmas can be found in \cite[II.70.1]{RogWil94} and \cite[IV.37.11]{RogWil94}, respectively.

\begin{lemma}\label{lem:M1}
% { \rm (Doob \cite[II.70.1]{RogWil94}.)} 
Let $\gb{Z(t)}_{t \ge 0}$ be a non-negative continuous martingale and $c > 0$. For any $t \ge 0$, we have
\begin{equation}\BP\left(\sup_{0 \le \th \le t} Z(\th) \ge c \right) \le \frac{1}{c}\,\BE Z(0).\end{equation}
\end{lemma}

\begin{lemma}\label{col:M3}
% {\rm \cite[IV.37.11]{RogWil94}.} 
Let $(b(t))_{t \ge 0}$ be an adapted measurable process such that
\begin{equation}\BE\,\int_0^t b(s)^2\,\mathrm ds < \infty\qquad \hbox{for every } t \ge 0.\end{equation}
The process $\gb{M(t)}_{t \ge 0}$ defined by
\begin{equation}\label{eq:M1}
M(t) = e^{\int_0^t b(s)\,\mathrm dW(s) - \frac{1}{2}\int_0^t b(s)^2\,\mathrm ds},\qquad t \ge 0,
\end{equation}
is a continuous martingale.
\end{lemma}

Below we state    important inequalities that will be used extensively in this paper.  The power of the Burkholder--Davis--Gundy inequality lies in the   values other than $p\neq 2.$  Indeed,  $p=2$ is a special case with  $c_2=1$ and $C_2=4$,  following quite directly from    Doob's maximal inequality. 

\begin{lemma}\label{lem:stochineq}
Let $(M(t))_{t\ge 0}$ be a continuous martingale on $(\Omega,\mathcal{F},(\mathcal{F}_t)_{t},\BP)$.
\begin{enumerate}
\item[\textup{(a)}] \textup{(Doob's maximal inequality.)} For any $p,q>1$ with $\frac{1}{p}+\frac{1}{q}=1$, we have
\begin{equation}\left(\BE \sup_{0\le s\le t} |M(s)|^p\right)^{1/p} \le q  \Bigl( \BE\,|M(t)|^p\Bigr)^{1/p}.\end{equation}
\item[\textup{(b)}] If $(M(t))_{t\ge 0}$ is bounded above in probability, then 
\begin{equation}\gb{\sup_{0\le s\le t} M(s)}_{t\ge 0}\end{equation}
is bounded above in probability.
\item[\textup{(c)}] \textup{(Burkholder--Davis--Gundy inequality.)} For every $p\geq 1$, we have
\begin{equation}c_{p} \mathbb{E}\left[[M]_t^{p / 2}\right] \leq  \mathbb E\left[\sup _{0 \leq s \leq t} |M(s)|^p\right] \leq C_{p} \mathbb{E}\left[[M]_{t}^{p / 2}\right],\quad t\geq 0,\end{equation}
where $([M]_t)_{t\geq 0}$ denotes the quadratic variation process of $M$ and with $c_p$ and $C_p$   constants independent of $M$ and $(\Omega,\mathcal{F},(\mathcal{F}_t)_{t\geq 0},\BP)$.
% \todo{Niet beter om gewoon de BDG voor sup over interval $[0,t]$ zodat je niet nodig hebt dat de quadratic variation een limiet heeft?}
\item[\textup{(d)}] For $p\ge 1$, any standard Brownian motion on $(\Omega,\mathcal{F},(\mathcal{F}_t)_{t\geq 0},\BP)$, $0\le t_0\le t_1 $, and any locally bounded adapted process $(X(t))_{t_0\le t\le t_1}$ one has
\begin{equation}\BE\left| \int_{t_0}^{t_1} X(s)\,\mathrm dW(s)\right|^p \le C_p \BE \left( \int_{t_0}^{t_1} X(s)^2\, \mathrm ds\right)^{p/2},\end{equation}
where $C_p$ is as in \textnormal{(c)}.
\end{enumerate}
\end{lemma}

\begin{proof}
For (a) and (c), we refer to \cite[Thm. I.20]{book:protter} and \cite[Thm. 26.12]{book:kallenberg}, respectively. For (b), let $R>0$ and   use property (a) to obtain
\begin{align}
\BP\, \gb{\sup_{0\le s\le t} M(s)\ge R} &= \BE\, \mathbbm{1}_{(R,\infty)}\Bigl(\sup_{0\le s\le t} M(s)\Bigr)^2\nonumber\\
\nonumber &\le 4 \sup_{0\le s\le t} \BE\, \mathbbm1_{[R,\infty)}(M(s))^2\\
&= 4 \sup_{0\le s\le t} \BP\gb{M(s)\ge R},\qquad\mbox{for all } t\ge 0.
\end{align}
So if $(M(t))_{t\ge 0}$ is bounded above in probability, then so is $(\sup_{0\le s\le t} M(s))_{t\ge 0}.$
Finally, for (d),  let $X'(t)=X(t)$ for $t\in [t_0,t_1]$ and $X'(t)=0$ for $t\in[0,T]\setminus [t_0,t_1]$, because then (c) yields
\begin{align}
\nonumber \BE \left| \int_{t_0}^{t_1} X(s)\,\mathrm  dW(s)\right|^p &\le \BE \left(\sup_{0\leq t\leq T} \left| \int_0^t X'(s)\,\mathrm  dW(s) \right|\right)^p\\
% \nonumber&\le c_p \BE \left[\int_0^\cdot X'(s)\,\mathrm  dW(s)\right] _\infty^{p/2}\\
\nonumber&\leq C_p \BE\,\left( \int_0^T X'(s)^2\, \mathrm ds\right)^{p/2}\\
&\le C_p \BE \left( \int_{t_0}^{t_1} X(s)^2\, \mathrm ds\right)^{p/2},
\end{align}
proving the final assertion.
\end{proof}

\section{Estimates for Itô processes with negative drift}\label{sec:estims}
\setcounter{equation}{0}
Let $(\Om,\SF,(\SF_t)_{t \ge 0},\BP)$ be a filtered probability space satisfying the usual conditions together with a Brownian motion $W=({W(t)})_{t \ge 0}$. 
We provide some estimates for the Itô process
\begin{equation}\label{eq:EsItô2}
Y(t) = -\int_0^t a(s)\,\mathrm ds + \int_0^t b(s)\,\mathrm dW(s),\qquad t \ge 0,
\end{equation}
 where $(a(t))_{t \geq 0}$ and $(b(t))_{t \geq 0}$ are measurable and adapted processes such that
 \begin{equation}
 \dint {0} {t}{|a(s)|}   s<\infty \mbox{ a.s.}     \quad\text{and}\quad  \BE\dint {0} {t}{b(s)^{2}}  s<\infty,  
 \end{equation}
 for all $t \geq 0$. In this setting, the It\^o process in \eqref{eq:EsItô2} is well-defined and its second term is a continuous martingale satisfying the It\^o isometry
 \begin{equation}
\mathbb{E} \left( \int_0^t b(s)\,\mathrm{d}W(s)\right)^2 =\mathbb{E} \dint {0} {t}{|b(s)|^2}   s    \qquad\text{for all }t\geq 0. 
 \end{equation}
 
The first two estimates rely on standard techniques and serve to illustrate how negative drift leads to exponential bounds on tail probabilities.

\begin{lemma}\label{lem:EsItô1}%E1
	Let $Y=(Y(t))_{t\geq 0}$ be a stochastic process given by \eqref{eq:EsItô2}. 
If there exists a constant $\be > 0$ such that we have
\begin{equation}\label{eq:EsItô3}
\int_0^t a(s)\,\mathrm ds \ge \be \int_0^t b(s)^2\,\mathrm ds\qquad \mbox{for all }t\ge 0\mbox{ a.s.},
\end{equation}
then for every $t \ge 0$ and $R \in \BR$ we have
\begin{equation}%\label{eq:EsItô4}
\BP\left(\sup_{0 \le \th \le t} Y(\th) \ge R\right) \le \exp\gb{-2\be R}.
\end{equation}
\end{lemma}

\begin{proof}
According to Lemma \ref{col:M3},
\begin{equation}\left(\exp\gb{2\be \int_0^t b(s)\,\mathrm dW(s) - 2\be^2 \int_0^t b(s)^2\,\mathrm ds}\right)_{t \ge 0}\end{equation}
is a non-negative continuous martingale. This implies, with the aid of Lemma \ref{lem:M1}, that for $t \ge 0$ and $R \in \BR$ we have  
\begin{align*}
\BP\,\left(\sup_{0 \le \th \le t} Y(\th) \ge R\right) &= \BP\,\left(\sup_{0 \le \th \le t} \exp\gb{2\be\int_0^\th b(s)\,\mathrm dW(s) - 2\be\int_0^\th a(s)\,\mathrm ds} \ge e^{2\be R}\right)\\
&\le  \BP\,\left(\sup_{0 \le \th \le t} \exp\gb{2\be\int_0^\th b(s)\,\mathrm dW(s) - 2\be^2\int_0^\th b(s)^2\,\mathrm ds} \ge e^{2\be R}\right)\\
&\le e^{-2\be R}\, \BE e^0 = e^{-2\be R},\yesnumber
\end{align*}
where we have exploited  \sef{eq:EsItô3} to estimate $\exp\gb{-2\be\gb{\int_0^\th a(s)\,\mathrm ds - \be \int_0^\th b(s)^2\,\mathrm ds}} \ge 1.$
\end{proof}

\begin{corollary}\label{col:EsItô2}
Let $\al >  0$ and $\si \in \BR$. Then, for every $t \ge 0$ and $R \in \BR$, we have 
\begin{equation}%\label{eq:EsItô5}
\BP\,\left(\sup_{0 \le \th \le t} \gb{-\al \th + \si W(\th)} \ge R \right) \le \exp\left({-\frac{2\al}{\si^2}R}\right).
\end{equation}
\end{corollary}

\begin{proof}
With $\be = \al\si^{-2}$ we have $\al t \ge \be \si^2 t$ for all $t \ge 0$ and Lemma \ref{lem:EsItô1} yields the result.
\end{proof}

Finding good probability estimates for the \textit{reverse time supremum}
 \begin{equation}\sup_{0 \le \th \le t} \gb{Y(t) - Y(\th)}  \end{equation}
is rather involved and the subject of the next lemmas.  In short, the proof of the first lemma requires a negative drift---no matter how small---to ensure an upper bound as in   \eqref{eqlem:EsItô7}. This bound tends to infinity as $\alpha \downarrow 0.$ 

\begin{lemma} \label{lem:EsItô2}%E4
	Let $Y=(Y(t))_{t\geq 0}$ be a stochastic process given by \eqref{eq:EsItô2}. Assume that there exist positive constants $\alpha$, $\beta_{0}$, and $\beta_{1}$ such that
	\begin{equation}\label{lem:EsItô6}
		a(s) \geq \alpha \quad \text { and } \quad 0<\beta_{0}^{2} \leq b(s)^{2} \leq \beta_{1}^{2}\quad  a.s.\, for\, all \,s\geq 0.
	\end{equation}
	Then for every $l \in \mathbb{N}$ and $R \geq 0$ we have
	\begin{equation}\label{eqlem:EsItô7}
		\mathbb{P}\left(\sup _{0 \leq \theta \leq l}(Y(l)-Y(\theta)) \geq R\right) \leq 2\exp \left(-\frac{R^{2}}{8 \beta_{1}^{2}}\right)+\frac{2\exp \left(-\frac{\alpha R}{8 \beta_{1}^{2}}\right)}{1-\exp \left(-\frac{\alpha^{2}}{16 \beta_{1}^{2}}\right)}.
	\end{equation}
\end{lemma}

\begin{proof}
	 Define
	\begin{equation}
	T(t):=\int_{0}^{t} b(s)^{2} \mathrm d s, \quad t \geq 0.
	\end{equation}
	Almost surely, the map $t \mapsto T(t)$ is continuous and strictly increasing, $T(0)=0,$ and $\lim _{t \rightarrow \infty} T(t)=\infty .$  Introduce the family of $\mathbb F$-stopping times $(S(t))_{t\geq 0}$ defined by
	\begin{equation}
		S(t)=\inf \left\{s \geq 0 ;\, T(s)>t\right\},
	\end{equation}
where $\mathbb{F}:=(\mathcal{F}_{t} )_{t \geq 0}$. Put
	$
	\mathcal{G}_{t}:=\mathcal{F}_{S(t)},  \, t \geq 0,
	$
	and consider the stochastic process $B=(B(t))_{t\geq0 } $, defined by
	\begin{equation}
	B(t):=\int_{0}^{S(s)} b(s) \,\mathrm d W(s), \quad t \geq 0.
	\end{equation}
	Thanks to the Dambis--Dubins--Schwarz time change result for local martingales, see Theorem \ref{thm:TC1}, we obtain that $B$ is a standard Brownian motion on $ (\Omega, \mathcal{F},\mathbb G, \mathbb{P} ) ,$ where $\mathbb G:=(\mathcal{G}_{t} )_{t \geq 0}$.
	
	For $l \in \mathbb{N}$ and $R \geq 0$ fixed, we have
	\begin{align} 
		\mathbb{P}\left(\sup _{0 \leq \theta \leq l}(Y(l)-Y(\theta)) \geq R\right)  
		\nonumber&=\mathbb{P}\left(\sup _{0 \leq \theta \leq t}-\int_{\theta}^{l} a(s) \,\mathrm d s+\int_{\theta}^{l} b(s) \,\mathrm d W(s) \geq R\right) \\
		\nonumber&\leq \mathbb{P}\left(\sup _{0 \leq \theta \leq t} \int_{\theta}^{l} b(s)\,\mathrm  d W(s)-\alpha(l-\theta) \geq R\right) \\
		\nonumber&=\mathbb{P}\left(\sup _{0 \leq \theta \leq t} B(T(l))-B(T(\theta))-\alpha(l-\theta) \geq R\right) \\
%		&=\mathbb{P}\left(\sup _{0 \leq \theta \leq t} \exists k \in[1, \ldots, l]: \sup _{k-1 \leq \theta \leq k} B(T(l))-B(T(\theta))-\alpha(l-\theta) \geq R\right) \\
		\nonumber&\leq \sum_{k=1}^{l} \mathbb{P}\left(\sup _{k-1 \leq \theta \leq k} B(T(l))-B(T(\theta))-\alpha(l-\theta) \geq R\right) \\
		&\leq \sum_{k=1}^{l} \mathbb{P}\left(\sup _{k-1 \leq \theta \leq k} B(T(l))-B(T(\theta)) \geq R+\alpha(l-k)\right).
	\end{align}
	For $k \in\{1, \ldots, l\},$ the time $k-1$ is a deterministic stopping time, hence Lemma \ref{lem:TC1} yields that $T(k-1)$ is a $\mathbb G$-stopping time. Indeed, for all $u\geq 0,$ we have
	$\{T(k-1)\leq u\}=\{k-1\leq S(u)\}\in \mathcal F_{k-1}\cap \mathcal F_{S(u)}$. According  to  Lemma \ref{lem:BM1}, it follows that
	\begin{equation}
	B_{k}(t):=B(T(k-1)+t)-B(T(k-1)), \quad t \geq 0 
	\end{equation}
	is a standard Brownian motion with respect to the filtration $\left(\mathcal{G}_{T(k-1)+t}\right)_{t \geq 0}$. 

Moreover, for any  $0 \leq s \leq t$ and  $k \in\{1, \ldots, l\},$ it holds that
	$ 
	T(t)-T(s) \leq \beta_{1}^{2}(t-s)
	$,  
	hence
	\begin{equation}
	T(l)-T(k-1) \leq \beta_{1}^{2}(l-k+1),
	\end{equation}
	and for every $\theta \in[k-1, k]$, we have
\begin{equation}	T(\theta)-T(k-1) \leq \beta_{1}^{2}(\theta-k+1) \leq \beta_{1}^{2}.\end{equation}
Consequently, we obtain the following inequality:
 \begin{align} 
	\sup _{k-1 \leq \theta \leq k} \nonumber & B(T(l))-B(T(\theta)) 
	\\\nonumber &\leq  B(T(l))-B(T(k-1))+\sup _{k-1 \leq \theta \leq k}-\big(B(T(\theta))-B(T(k-1))\big) \\
	\nonumber &=  B_{k}(T(l)-T(k-1))+\sup _{k-1 \leq \theta \leq k} \big(-B_{k}(T(\theta)-T(k-1))\big) \\
	\nonumber &\leq  \sup _{0 \leq \theta \leq \beta_{1}^{2}(l-k+1)} B_{k}(\theta)+\sup _{0 \leq \theta \leq \beta_{1}^{2}}\big(-B_{k}(\theta)\big) \\
	&\leq     \sup _{0 \leq \theta \leq \beta_{1}^{2}(l-k+1)} B_{k}(\theta)+   \sup _{0 \leq \theta \leq \beta_{1}^{2}(l-k+1)} \big(-B_{k}(\theta)\big).
\end{align}
Note that $-B_k=(-B_k(t))_{t\geq 0}$ is a standard Brownian motion, since $B_k$ is a standard Brownian motion. 
%For our own convenience, introduce the notation
%\begin{equation}
%	X_k^+(t)= \sup _{0 \leq \theta \leq t} B_{k}(\theta)\quad\text{and}\quad X_k^-(t)=\sup %_{0 \leq \theta \leq t} \big(-B_{k}(\theta)\big) ,\quad t\geq 0.
%\end{equation}
%By the reflection principle for Brownian motion \cite[Thm I.33]{book:protter}, we obtain
Due to Corollary \ref{col:BM1} we have
\begin{equation}
\mathbb{P} \left(\sup _{0 \leq \theta \leq t} B_{k}(\theta)\geq c\right) = \mathbb{P} \left(\sup _{0 \leq \theta \leq t} \left(-B_{k}(\theta)\right)\geq c\right)  \leq \exp\left(-\frac{c^2}{2t}\right),%\label{A11}
\end{equation}
for any $t> 0$ and $c\geq 0.$ 
%The inequality in \eqref{A11} is easily  obtained via  a change of variables, where one considers the transformation   $u=x-c.$ 
%The same upper bound  holds true for $X_k^-$, simply due to the fact that $X_k^+$ and $X_k^-$ equal in distribution.  

Combining all the results above yields
\begin{align}
	\nonumber\mathbb{P}\Bigg(\sup _{0 \leq \theta \leq l}-\int_{\theta}^{l} \nonumber&a(s) \,\mathrm d s+\int_{\theta}^{l} b(s)\,\mathrm d W(s) \geq R\Bigg)\\   
	\nonumber&	\leq \sum_{k=1}^{l} \mathbb{P}\left(\sup _{k-1 \leq \theta \leq k} B(T(l))-B(T(\theta)) \geq  \alpha(l-k)+R\right) \\
	\nonumber&\leq  \sum_{k=1}^{l} \mathbb{P}\Big( 
    \sup_{0\le \theta\le \be_1^2(l-k+1)} B_k(\theta)
   + \sup_{0\le \theta\le \be_1^2(l-k+1)} \left(-B_k(\theta)\right)  \geq \alpha(l-k)+R\Big) \\
	\nonumber&\leq  \sum_{k=1}^{l} 2\cdot \mathbb{P}\left(\sup_{0\le \theta\le \be_1^2(l-k+1)} B_k(\theta) \geq \frac{\alpha}{2}(l-k)+\frac{1}{2} R\right) \\
	&\leq  2\sum_{k=1}^{l} \exp \left(-\frac{\left(\frac{\alpha}{2}(l-k)+\frac{1}{2} R\right)^{2}}{2 \beta_{1}^{2}(l-k+1)}\right).
\end{align} 
 The assertion now follows from the estimation  
\begin{align}
	\sum_{k=1}^{l} \exp \nonumber&\left(-\frac{\left(\frac{\alpha}{2}(l-k)+\frac{1}{2} R\right)^{2}}{2 \beta_{1}^{2}(l-k+1)}\right) 
	=   \sum_{k=0}^{l-1} \exp \left(-\frac{(\alpha k+R)^{2}}{8 \beta_{1}^{2}(k+1)}\right) \\
\nonumber&\qquad\qquad	= \sum_{k=0}^{l-1} \exp \left(-\frac{\alpha^{2} k^{2}+2 \alpha k R+R^{2}}{8 \beta_{1}^{2}(k+1)}\right) \\
\nonumber&\qquad\qquad	= \sum_{k=0}^{l-1} \exp \left(-\frac{\alpha^2 }{8 \beta_{1}^{2}} \frac{k^{2}}{k+1}-\frac{2 \alpha R}{8 \beta_{1}^{2}} \frac{k}{k+1}\right) \exp \left(-\frac{R^{2}}{8 \beta_{1}^{2}(k+1)}\right) \\
 \nonumber&\qquad\qquad	\leq \exp \left(-\frac{R^{2}}{8 \beta_{1}^{2}}\right)+\sum_{k=1}^{l-1} \exp \left(-\frac{\alpha^{2}}{16 \beta_{1}^{2}}k-\frac{\alpha R}{8 \beta_{1}^{2}}\right) \\
 &\qquad\qquad 	\leq\exp \left(-\frac{R^{2}}{8 \beta_{1}^{2}}\right)+\exp \left(-\frac{\alpha R}{8 \beta_{1}^{2}}\right)\left(1-\exp \left(-\frac{\alpha^{2}}{16 \beta_{1}^{2}}\right)\right)^{-1},
\end{align}
where we have used that $k /(k+1) \geq 1 / 2$ holds for $k\geq 1$. 
\end{proof}

The condition in Lemma \ref{lem:EsItô2} that $b$ be nonzero can be disposed of at the cost of larger constants in the estimate. The following slight generalisation in this spirit is  useful in the sequel.

	 \begin{lemma}\label{lem:EsItô2bis}
	 	 	Let $Y=(Y(t))_{t\geq 0}$ be a stochastic process given by \eqref{eq:EsItô2}.  
	 	 	 Assume that there exist positive constants $\alpha$ and $\beta$ such that
	 	 	\begin{equation}
	 	 		a(s) \geq \alpha \quad \text { and } \quad     b(s)^{2} \leq \beta^{2}\quad   \hbox{a.s. for all } s\geq 0. 
	 	 	\end{equation}
	Then for every $l \in \mathbb{N}$ and $R \geq 0$ we have
	\begin{equation}\label{eq:EsItô7}
	\mathbb{P}\left(\sup _{0 \leq \theta \leq l}(Y(l)-Y(\theta)) \geq R\right) \leq 4 \exp \left(-\frac{R^{2}}{64 \beta^{2}}\right)+\frac{4\exp \left(-\frac{\alpha R}{64 \beta^{2}}\right)}{1-\exp \left(-\frac{\alpha^{2}}{128 \beta^{2}}\right)}.
	\end{equation}
	 \end{lemma}
	
\begin{proof} Let $c>\beta$ be arbitrary, yet fixed. Introduce the stochastic processes $Y_1$ and $Y_2$, where
	\begin{equation}
		Y_{1}(t):=-\frac{1}{2} \int_{0}^{t} a(s) \,\mathrm d s+\frac{1}{2} \int_{0}^{t}(b(s)+c) \,\mathrm d W(s),\quad t\geq 0,  
	\end{equation}
and
\begin{equation}
		Y_{2}(t):=-\frac{1}{2} \int_{0}^{t} a(s) \,\mathrm d s+\frac{1}{2} \int_{0}^{t}(b(s)-c) \,\mathrm d W(s),\quad t\geq 0.
	\end{equation} 
	  Note that $Y(t)=Y_{1}(t)+Y_{2}(t)$
	 holds for all $t \geq 0$.
	 
	  In addition, the stochastic processes $Y_{1}$ and $Y_{2}$ satisfy the conditions of Lemma \ref{lem:EsItô2} with $\alpha$ and $\beta^{2}_1$ replaced by $\alpha / 2$ and   $ (\beta  +c)^{2} / 4$, respectively, since
	\begin{equation}
\beta+c\geq 	b(s)+c \geq-\beta  +c>0 \quad \text { and } \quad -(\beta+c)\leq   b(s)-c<\beta -c<0.
	\end{equation}
	Taking $c=(\sqrt8-1) \beta >\beta$ specifically, gives us $ (\beta  +c)^{2} / 4=2\beta^2$ and
	\begin{align} 
		&\nonumber\mathbb{P}\left(\sup _{0 \leq \theta \leq l}(Y(l)-Y(\theta)) \geq R\right) \\ 
		 \nonumber&\qquad\qquad\leq \mathbb{P}\left(\sup _{0 \leq \theta \leq l}\left(Y_{1}(l)-Y_{1}(\theta)\right)+\sup _{0 \leq \theta \leq l}\left(Y_{2}(l)-Y_{2}(\theta)\right) \geq R\right) \\
		\nonumber&\qquad\qquad \leq \mathbb{P}\left(\sup _{0 \leq \theta \leq l}\left(Y_{1}(l)-Y_{1}(\theta)\right) \geq R / 2\right)+\mathbb{P}\left(\sup _{0 \leq \theta \leq l}\left(Y_{2}(l)-Y_{2}(\theta)\right) \geq R / 2\right) \\
		&\qquad\qquad \leq 4 \exp \left(-\frac{R^{2}}{64 \beta^{2}}\right)+\frac{4\exp \left(-\frac{\alpha R}{64 \beta^{2}}\right)}{1-\exp \left(-\frac{\alpha^{2}}{128 \beta^{2}}\right)}.
	\end{align}
This completes the proof.
\end{proof}

An important feature of the estimates in Lemma \ref{lem:EsItô1} and Corollary \ref{col:EsItô2}  is that they do not depend on the length of the interval $[0,l]$ over which the supremum is taken. It is easier to obtain estimates for a supremum over an interval with a fixed length, for which one does not a negative drift. Such estimates are known even in a setting of Banach space valued stochastic integrals; see \cite[Corollary 4.4]{vanNeerven-Veraar2020}. We present an estimate in the next lemma which is sufficient for our purposes and we give a proof using the same techniques as the previous proofs in this section. See \cite[Lemma 4.3]{vandenbosch2026b} for an alternative proof and a more detailed discussion.

 \begin{lemma}\label{lem:EsItô3a}%E3-New
  	Let $W=(W(t))_{t\geq 0}$ be a standard Brownian motion. Assume there exists a positive constant $\beta$ such that
  	\begin{equation}
  		b(s)^2\leq \beta^2\quad \text{a.s.  for all $s\geq 0$}.
  	\end{equation}
  	Let $t_0\geq 0$ and $T>0$ be fixed. Then for every $R\geq 0$ we have
  	\begin{equation}
  		\mathbb P\left(\sup_{t_0\leq t\leq t_0+T}\dint{t_0}t{b(s)}W(s)\geq R\right)\leq 2\exp\left(-\frac{R^2}{16\beta^2T}\right).
  	\end{equation}
        If, in addition, there exists a positive constant $\beta_0$ such that
  	\begin{equation}
  		0 < \beta_0^2 \le b(s)^2\leq \beta^2\quad \text{a.s.  for all $s\geq 0$},
  	\end{equation}
  	then
  	\begin{equation}
  		\mathbb P\left(\sup_{t_0\leq t\leq t_0+T}\dint{t_0}t{b(s)}W(s)\geq R\right)\leq \exp\left(-\frac{R^2}{2\beta^2T}\right).
  	\end{equation}
  \end{lemma}

 \begin{proof}
  	As before,  first consider the case that there exists  positive constants $\beta$ and $\beta_0$ such that
 \begin{equation}\label{eq:coeff_b_strictly_positive}
   0<\beta_{0}^{2} \leq b(s)^{2} \leq \beta^{2}\quad \hbox{a.s. for all } s\geq 0. 
 \end{equation} 
	Define the continuous martingale
	\begin{equation}
	M(t):=\int_{t_{0}}^{t_{0}+{t}} b(s)\,\mathrm  d W(s), \quad t \geq 0, 
	\end{equation}
	which is with respect to the filtration $\tilde {\mathbb F}:= ( \tilde{\mathcal{F}}_{t} )_{t \geq 0}$, given by
	$ 
	\tilde{\mathcal{F}}_{t}:=\mathcal{F}_{t_{0}+t},\,   t \geq 0.
	$ 
Subsequently, consider the time change
% \todo{2x T onhandig}
	\begin{equation}
	T(t)=\int_{t_{0}}^{t_{0}+t} b(s) ^2\,\mathrm d s, \quad t \geq 0.
\end{equation}
	The map $t \mapsto T(t)$ is continuous and strictly increasing a.s., $T(0)=0,$ and $\lim _{t \rightarrow \infty} T(t)=\infty .$

Let $(S(t))_{t\geq 0}$ be  the $\tilde{\mathbb F}$-stopping times defined by
		$
		S(t)=\inf \left\{s \geq 0 ;\, T(s)>t\right\}.
	$
	Set
	$
	\mathcal{G}_{t}:=\tilde {\mathcal{F}}_{S(t)}$,  $ t \geq 0.
	$
		Thanks to the  Dambis--Dubins--Schwarz theorem, see Theorem \ref{thm:TC1}, we know that there exists a standard Brownian motion $B$ on $ (\Omega, \mathcal{F},\mathbb G, \mathbb{P} ) ,$ where $\mathbb G:=(\mathcal{G}_{t} )_{t \geq 0}$,  such that
	\begin{equation}
	\int_{t_{0}}^{t_{0}+t} b(s) \,\mathrm d W(s)=B(T(t)),  \quad t \geq 0.
	\end{equation}
	For every $t \in\left[t_{0}, t_{0}+T\right]$ we have
	\begin{equation}
	T(t) \geq 0 \quad \text { and } \quad T(t) \leq \beta^{2} T,
	\end{equation}
	holding $\mathbb P $-almost surely.
	Hence
	\begin{align}
		\nonumber\mathbb{P}\left(\sup _{t_{0} \leq t \leq t_{0}+T} \int_{t_{0}}^{t} b(s) \,\mathrm d W(s) \geq R\right) &=\mathbb{P}\left(\sup _{0 \leq t \leq T} B(T(t)) \geq R\right) \\
		\nonumber& \leq \mathbb{P}\left(\sup _{0 \leq t \leq \beta^{2} T} B(t) \geq R\right) \\
		& \leq \exp \left(-\frac{R^{2}}{2 \beta^{2} T}\right),\label{eq:exp_estim_end_proof_lemma}
	\end{align} 
by Corollary \ref{col:BM1}. 
	
We can reduce the general case to the case where \eqref{eq:coeff_b_strictly_positive} holds. This is achieved along    similar  lines as in the proof of Lemma \ref{lem:EsItô2bis}. In essence, we can choose $c>\beta$ such that $0<c\pm b(s)\le 2\sqrt{2}\beta$ and  apply \eqref{eq:exp_estim_end_proof_lemma} twice with $\beta^{2}$ replaced by $8\beta^{2}$. 
\end{proof}

\begin{corollary}\label{col:EsItô3abis}
Let $Y(t) $ be given by \sef{eq:EsItô2}. Assume there exist positive constants $\al$ and $\be$ such that
\begin{equation}a(s) \ge \al\quad\hbox{and}\quad b(s)^2 \le \be^2\quad \hbox{a.s. for all } s \ge 0.\end{equation}
Then the process
\begin{equation}\left(\sup_{0 \le \th \le t} \gb{Y(t) - Y(\th)} \right)_{t \ge 0}\end{equation}is bounded above in probability.
\end{corollary}

\begin{proof}
Let $R > 0$ and let $t \ge 0$. Let $l$ be the greatest integer below $t$. For $t_0 \in [0,l]$ we have that the expression
\begin{equation}\label{eq:EsItô3abis}
Y(t) - Y(t_0) = -\int_{t_0}^t a(s)\,\mathrm ds + \int_{t_0}^t b(s)\,\mathrm dW(s)
\end{equation}
is at most
\begin{equation}-\int_{t_0}^l a(s)\,\mathrm ds + \int_{t_0}^l b(s)\,\mathrm dW(s) + \int_{l}^t b(s)\,\mathrm dW(s),\end{equation}
and for $t_0 \in [l,t]$ we have that \sef{eq:EsItô3abis} is at most
\begin{equation}\int_{t_0}^t b(s)\,\mathrm dW(s) \le \sup_{\th \in [(t-1)^+,t]} \int_{\th}^t b(s)\,\mathrm dW(s),\end{equation}
where $(t-1)^+=t-1$ if $t\ge 1$ and $(t-1)^+=0$ if $t<1$.
As $l\in [(t-1)^+,t]$, the first case yields for every $t_0\in [0,l]$ that \sef{eq:EsItô3abis} is at most
\begin{equation}\label{eq:EsItô3abis-2}
	\sup_{\th \in [0,l]}\left(-\int_{\th}^l a(s)\,\mathrm ds + \int_{\th}^l b(s)\,\mathrm dW(s) \right) + \sup_{\th \in [(t-1)^+,t]} \int_{\th}^t b(s)\,\mathrm dW(s).
\end{equation}
The second case yields that \sef{eq:EsItô3abis} is also below \sef{eq:EsItô3abis-2} for every $t_0\in [l,t]$.
Combining both cases we obtain
\begin{align}
\nonumber&\BP\,\left(\sup_{t_0 \in [0,t]} Y(t) - Y(t_0) \ge R \right)\\
\nonumber&\qquad\le \BP\,\left(\sup_{\th \in [0,l]}\left(-\int_{\th}^l a(s)\,\mathrm ds + \int_{\th}^l b(s)\,\mathrm dW(s) \right) + \sup_{\th \in [(t-1)^+,t]} \int_{\th}^t b(s)\,\mathrm dW(s) \ge R \right)\\
\nonumber&\qquad\le \BP\,\left(\sup_{\th \in [0,l]}\left(-\int_{\th}^l a(s)\,\mathrm ds + \int_{\th}^l b(s)\,\mathrm dW(s)\right) \ge R/2 \right)\\
&\qquad\qquad\qquad +
\BP\,\left(\sup_{\th \in [(t-1)^+,t]} \int_{\th}^t b(s)\,\mathrm dW(s) \ge R/2 \right),
\end{align}
which is below a given $\ep > 0$ for all $t \ge 0$, provided $R$ is chosen big enough, due to Lemma \ref{lem:EsItô2bis} and Lemma \ref{lem:EsItô3a}.
\end{proof}

In the next section, we will also require three estimates on sums of probabilities involving the process $Y$, stated in the Lemmas \ref{lem:EsItô3}, \ref{lem:EsItô4}, and \ref{lem:EsItô5} below. The first two of them are the most delicate ones, yet make use of the following straightforward calculation.

\begin{lemma}\label{lem:E5}
	For every $p,q \in \BR$ with $0 < p < q$ we have
	\begin{equation}\sum_{l=0}^\infty \frac{1}{(pl+q)^2} \le \frac{1}{p(q-p)}.\end{equation}
\end{lemma}

\begin{proof}
	For $c > 1$ we have	\begin{equation}\sum_{l=0}^\infty \frac{1}{(l+c)^2} \le \int_{c-1}^\infty \frac{1}{x^2}\,\mathrm dx = \frac{1}{(c-1)},\end{equation}
yielding
	\begin{equation}
		\sum_{l=0}^\infty \frac{1}{(pl+q)^2}  = \frac{1}{p^2}\sum_{l=0}^\infty \frac{1}{(l+q/p)^2}
		\le \frac{1}{p^2(q/p-1)} =  \frac{1}{p(q-p)},
	\end{equation}
    proving the estimate.
\end{proof}

\begin{lemma}\label{lem:EsItô3}%E6
Let $Y(t) $ be given by \sef{eq:EsItô2} and let $p$ and $C$ be given positive constants. Assume there exist positive constants $\al$, $\lambda$,  and $\si$ such that  
% \todo{niet nodig dat sigma 0 niet 0 mag zijn toch?}
\begin{equation}\label{eq:EsItô6}
a(s) \ge \al\quad\hbox{and}\quad b(s)^2 \le \si^2  \quad\hbox{a.s. for all } s \ge 0
\end{equation}
and
\begin{equation}\al > 32\lambda \si^2.\end{equation}
For every $\ep > 0$ there exists $q > 0$ such that for every $m \in \BN$ we have
\begin{equation}%\label{eq:EsItô10a}
\sum_{l=0}^m \BP\,\left(\sup_{0 \le \th \le m-l} \gb{Y(m-l) - Y(\th)} \ge R_l \right) < \ep,
\end{equation}
where
\begin{equation}R_l = - C + \tfrac{1}{2\lambda}\log \gb{pl+q}.\end{equation}
\end{lemma}

\begin{proof}
We prove the lemma for $\lambda =1/2$. The general case follows by rescaling $Y$.
According to Lemma \ref{lem:EsItô2bis}, we have
% \todo{Lemma 3.4 zodat we ook $\sigma_0$ lekker ook 0 kunnen kiezen}
\begin{align}
\nonumber&\sum_{l=0}^m \BP\,\left(\sup_{0 \le \th \le m-l} \gb{Y(m-l) - Y(\th)} \ge R_l \right)\\
\nonumber&\qquad\le
4\sum_{l=0}^m\exp\gb{-\frac{R_l^2}{64\si^2}} + \frac{4}{1 - \exp\gb{-\frac{\al^2}{128\si^2}}} \sum_{l=0}^m \exp\gb{-\frac{\al R_l}{64\si^2}}\\
\nonumber&\qquad\le 4\sum_{l=0}^m \exp\left(-\frac{C^2 + \log(pl+q)\gb{\log (pl+q) - 2C}}{64\si^2}\right)\\
&\qquad\qquad\quad+\frac{4\exp\gb{\frac{\al C}{64\si^2}}}{1 - \exp\gb{-\frac{\al^2}{128\si^2}}} \sum_{l=0}^m \exp\gb{-\frac{\al \log (pl+q)}{64\si^2}}.
\end{align}
Next choose $q$ so large that $\log q - 2C \ge 128\si^2$ and proceed the estimation by
\begin{align}
\nonumber&\sum_{l=0}^m \BP\,\left(\sup_{0 \le \th \le m-l} \gb{Y(m-l) - Y(\th)} \ge R_l \right)\\
\nonumber&\qquad\le
4\left(1 + \frac{\exp\gb{\frac{\al C}{64\si^2}}}{1 - \exp\gb{-\frac{\al^2}{128\si^2}}}\right) \sum_{l=0}^m \exp\gb{-2 \log (pl+q)}\\
\nonumber&\qquad=
4\left(1 + \frac{\exp\gb{\frac{\al C}{64\si^2}}}{1 - \exp\gb{-\frac{\al^2}{128\si^2}}} \right) \sum_{l=0}^m \frac{1}{(pl+q)^2}\\
&\qquad\le
4\left(1 + \frac{\exp\gb{\frac{\al C}{64\si^2}}}{1 - \exp\gb{-\frac{\al^2}{128\si^2}}} \right) \frac{1}{p(q-p)},
\end{align}
where we have used Lemma \ref{lem:E5}. Summarizing this shows that
\begin{equation}
\sum_{l=0}^m \BP\,\left(\sup_{0 \le \th \le m-l} \gb{Y(m-l) - Y(\th)} \ge R_l \right)\le
4\left(1 + \frac{\exp\gb{\frac{\al C}{64\si^2}}}{1 - \exp\gb{-\frac{\al^2}{128\si^2}}} \right) \frac{1}{p(q-p)},
\end{equation}
which tends to zero as $q \to \infty$.
\end{proof}

Now that we have obtained estimates for $Y(t)$ at integer time points $t=l$, $l \in \BN$, the next step is to relate the process $Y(t)$ to $Y(l)$. In the lemma below, we  allow for positive drift as $a(t)$ may attain negative values.
% \todo{Nog iets opmerken dat hier geen negatieve drift nodig is, of juist positieve drift willen. Beetje in strijd met titel van sectie?!}

\begin{lemma}\label{lem:EsItô4}%E7
Let $Y(t) $ be given by \sef{eq:EsItô2} and $p$, $\lambda$, and $C$  positive constants. Assume that there exist  constants $\al \in \BR$ and $\si \in \BR$ such that 
% \todo{niet nodig dat sigma 0 niet 0 mag zijn toch?}
\begin{equation}
-a(s) \le \al\quad\hbox{and}\quad b(s)^2 \le \si^2\quad  \hbox{a.s. for all } s \ge 0.
\end{equation}
For every $\ep > 0$ there exists $q > 0$ such that for every $m \in \BN$ we have
\begin{equation}%\label{eq:EsItô10a}
\sum_{l=0}^m \BP\,\left(\sup_{m-l-1 \le t \le m-l+1} \gb{Y(t) - Y(m-l-1)} \ge R_l \right) < \ep,
\end{equation}
where
% \todo{in dit lemma, wat is $\lambda?$ $R_l$ heeft niet zoveel hier te maken eigenlijk?}
\begin{equation}R_l = - C + \tfrac{1}{2\lambda}\log (pl+q).\end{equation}
\end{lemma}

\begin{proof}
We prove the lemma for $\lambda =1/2$.
The general case again follows by rescaling $Y$.
Clearly, we may assume in the proof that $\al > 0$. Due to Lemma \ref{lem:EsItô3a}, we have
\begin{align}\label{eq:EsItô10ab}
&\BP \left(\sup_{m-l-1 \le t \le m-l+1} - \int_{m-l-1}^t a(s)\,\mathrm ds + \int_{m-l-1}^t b(s)\,\mathrm dW(s) \ge R_l \right)\nonumber\\
&\qquad\le
\BP \left(\sup_{m-l-1 \le t \le m -l+1} \int_{m-l-1}^t b(s)\,\mathrm dW(s) \ge R_l - 2 \al\right)\nonumber\\
&\qquad\le
2\exp\gb{-\frac{(R_l-2\al)^2}{16\si^2}}\nonumber\\
&\qquad = 2\exp\gb{-\frac{(-C-2\al + \log (pl+q))^2}{16\si^2}}\nonumber\\
&\qquad\le
2\exp\gb{-\frac{(C+2\al)^2 + \log(pl+q)\gb{\log (pl+q) - 2C -4\al}}{16\si^2}}.
\end{align}
Choose $q$ so large that $\log q > 32\si^2 + 2C + 4\al$ holds, sum the inequality \sef{eq:EsItô10ab} from $l=0$ to $m$, and use Lemma \ref{lem:E5} to arrive at
\begin{align}
\nonumber&\sum_{l=0}^m \BP \left(\sup_{m-l-1 \le t \le m-l+1} - \int_{m-l-1}^t a(s)\,\mathrm ds + \int_{m-l-1}^t b(s)\,\mathrm dW(s) \ge R_l \right)\\
\nonumber &\qquad\qquad\le
2\sum_{l=0}^m \exp\gb{-2\log (pl+q)}\\
&\qquad\qquad = 2\sum_{l=0}^m \frac{1}{(pl+q)^2} \le \frac{2}{p(q-p)},
\end{align}
which tends to $0$ as $q \to \infty$.
\end{proof}
    In view of the proof of the main theorem in Section \ref{sec:proof}, we  encounter a process $Y(t)$ which has a positive drift term. We may hence apply Lemmas \ref{lem:EsItô4} and \ref{lem:EsItô5} to it. Carefully observe the minus sign within \eqref{eq:minus} in front of the difference $Y(m+1)-Y(t)$. To remain in the spirit of this section, we  reformulate Lemma \ref{lem:EsItô5} in terms of processes $Y$ with negative drift, which is done in Corollary \ref{cor:reformulate}. After all,  $Y$ is an It\^o process with negative drift if and only if $-Y$ is one with   positive drift.

\begin{lemma}\label{lem:EsItô5}%E8
Let $Y(t) $ be given by \sef{eq:EsItô2} and $p>0$ a  positive constant. Assume that there exists a constant $\si \in \BR$ such that
% \todo{niet nodig dat sigma 0 niet 0 mag zijn toch? Hier dus niet-negatieve drift, correct? Ik zou dit wel duidelijk opmerken, hier kwam ik pas later achter.}
\begin{equation}
a(s) \le 0\quad\hbox{and}\quad b(s)^2 \le \si^2\quad  \hbox{a.s. for all } s \ge 0.
\end{equation}
For every $\ep > 0$ there exists $q > 0$ such that for every $m \in \BN$ we have
\begin{equation}%\label{eq:EsItô10a}
\sum_{l=0}^m \BP\,\left(\sup_{m-l \le t \le m-l+1} -\gb{Y(m+1) - Y(t)} \ge pl+q \right) < \ep.\label{eq:minus}
\end{equation}
\end{lemma}

\begin{proof}
Due to Lemma \ref{lem:EsItô3a}, we have for $q > 0$ the inequality
% \todo{Er stond hier een min foutje bij a en de constantes klopte niet.}
{\allowdisplaybreaks
\begin{align}
\nonumber&\BP\left(\sup_{m-l \le t \le m-l+1} -\gb{Y(m+1) - Y(t)} \ge pl+q \right)\\
\nonumber&\qquad\qquad\le 
\BP\left(\sup_{m-l \le t \le m-l+1} \int_t^{m+1} a(s)\,\mathrm ds - \int_t^{m+1} b(s)\,\mathrm dW(s) \ge pl+q\right)\\
\nonumber&\qquad\qquad\le 
\BP\left(\sup_{m-l \le t \le m-l+1}  - \int_t^{m+1} b(s)\,\mathrm dW(s) \ge pl+q\right)\\
\nonumber&\qquad\qquad\le 
\BP\left(\sup_{m-l \le t \le m-l+1}  - \int_{m-l}^{m+1} b(s)\,\mathrm dW(s) + \int_{m-l}^t b(s)\,\mathrm dW(s) \ge pl+q\right)\\
\nonumber&\qquad\qquad\le
\BP\left(- \int_{m-l}^{m+1} b(s)\,\mathrm dW(s) \ge \frac{1}{2}(pl+q) \right)\\
\nonumber&\qquad\qquad\qquad\qquad + 
\BP\left(\sup_{m-l \le t \le m-l+1} \int_{m-l}^t b(s)\,\mathrm dW(s) \ge \frac{1}{2}(pl+q) \right)\\
\nonumber&\qquad\qquad\le 
4\exp\gb{-\frac{(pl+q)^2}{64\si^2(l+1)}}\\
&\qquad\qquad = 
4\exp\gb{-\frac{q^2+2pql+p^2l^2}{64\si^2(l+1)}}.
\end{align}}
Summation from $l=0$ to $m$ yields
\begin{align}
\nonumber&\sum_{l=0}^m \BP\left(\sup_{m-l \le t \le m-l+1} -\gb{Y(m+1) - Y(t)} \ge pl+q \right)\\
\nonumber&\qquad\qquad\le 
4\exp\gb{-\frac{q^2}{64\si^2}} + \sum_{l=1}^m 4\exp\gb{-\frac{2pql}{64\si^2(l+1)}} \exp\gb{-\frac{p^2l^2}{64\si^2(l+1)}}\\
\nonumber&\qquad\qquad\le 
4\exp\gb{-\frac{q^2}{64\si^2}} + \sum_{l=1}^m 4\exp\gb{-\frac{pq}{64\si^2}} \exp\gb{-\frac{p^2l}{128\si^2}}\\
&\qquad\qquad\le 
4\exp\gb{-\frac{q^2}{64\si^2}} + 4\exp\gb{-\frac{pq}{64\si^2}}\frac{1}{1 - \exp\gb{-\frac{p^2}{128\si^2}}},
\end{align}
which tends to $0$ as $q \to \infty$.
\end{proof}
\begin{corollary}\label{cor:reformulate}%E8
Let $Y(t) $ be given by \sef{eq:EsItô2} and $p>0$ a  positive constant. Assume that there exists a constant $\si \in \BR$ such that
% \todo{niet nodig dat sigma 0 niet 0 mag zijn toch? Hier dus niet-negatieve drift, correct? Ik zou dit wel duidelijk opmerken, hier kwam ik pas later achter.}
\begin{equation}
a(s) \ge 0\quad\hbox{and}\quad b(s)^2 \le \si^2\quad  \hbox{a.s. for all } s \ge 0.
\end{equation}
For every $\ep > 0$ there exists $q > 0$ such that for every $m \in \BN$ we have
\begin{equation}%\label{eq:EsItô10a}
\sum_{l=0}^m \BP\,\left(\sup_{m-l \le t \le m-l+1} \gb{Y(m+1) - Y(t)} \ge pl+q \right) < \ep. 
\end{equation}
\end{corollary}

\section{Proof of main theorem}\label{sec:proof}
In this section we prove Theorem \ref{thm:Wrights} for the stochastic transformed Wright equation \eqref{eq:stoch-t-Wrights},
\begin{equation}
\mathrm dx(t) = -r\gb{e^{x(t-1)} - 1}\,\mathrm dt - a(x_t)\,\mathrm dt + b(x_t)\,\mathrm dW(t).
\end{equation} 
It turns out that our proof actually applies to a class of equations, which includes \eqref{eq:stoch-t-Wrights}.
Let $G:\mathbb R\to\mathbb R$ be a continuous function such that there exist $\gamma_0 \in \mathbb R$, $\gamma > 0$ and $\lambda > 0$  such that
\begin{equation}\label{as:G1}
0 \le G(x)\le \gamma_0 + \gamma e^{\lambda x} \mbox{ for all } x \in \mathbb R
\end{equation}
and
\begin{equation}\label{as:G2}
\lim_{x \to -\infty} G(x) = 0 \quad\mbox{ and }\quad \lim_{x \to \infty} G(x) = \infty.
\end{equation}
The  stochastic transformed Wright equation \eqref{eq:stoch-t-Wrights} is of the form
\begin{equation}\label{eq:genSTW}
\mathrm{d}z(t)=-G(z(t-1))\,\mathrm{d}t + (r-a(z_t))\,\mathrm{d}t + b(z_t)\,\mathrm{d}W(t),
\end{equation}
where $G(x)=re^x$, $0\le a(u)\le\alpha$, and $b(u)^2\le \beta^2$ for all $u\in C[-1,0]$, for some positive constants $\alpha$ and $\beta$. 
% Moreover, we have $r>\alpha_1$.
Due to Proposition \ref{lem:globlexistencepuredelay}, we have that for every $\mathcal{F}_0$-measurable random variable $\varphi$ in $C[-1,0]$, there exists a unique global solution $z$ of \eqref{eq:genSTW} with $z_0=\varphi$.

Our estimates in this section hold for processes satisfying the even more general equation 
% \todo{De twee a's zijn wel net verschillend, maakt het een beetje verwarrend, maar ik heb geen goede suggestie.}
% \todo{Ik heb hier $-a(t)$ van gemaakt. Nodig in Theorem 4.2.}
\begin{equation}\label{eq:genlized-transformed-stoch-Wright}%\label{eq:EsItô11}
		\mathrm dz(t) = -G(z(t-1))\,\mathrm dt +a(t)\,\mathrm dt + b(t)\,\mathrm dW(t), \qquad t \ge 0,
\end{equation}
where $(a(t))_{t\ge 0}$ and $(b(t))_{t\ge 0}$ are measurable and adapted processes satisfying $\alpha_0\le a(t)\le\alpha_1$, and $b(t)^2\le \si^2$ for all $t\ge 0$ for some positive constants $\alpha_0$, $\alpha_1$, and $\si$. We deal with system \eqref{eq:genlized-transformed-stoch-Wright} in Lemmas \ref{thm:upper-Wright-type}--\ref{lem:EsMain1}. This is followed by Theorem \ref{thm:Wrights2}, where we consider the system \eqref{eq:genSTW} and exploit the previous lemmas.
% , and $(b(t))_{t\ge 0}$ is adapted. \todo{Hier is al geen $\sigma_0.$}

But first, we start with a result on a general class of integral delay equations, which allows us to estimate solutions from above by increments of the driving process. It is similar to \cite[Lem. 5.3]{vandenbosch2026b}.

\begin{lemma}\label{lem.zbddabove}
	Let $f\colon \mathbb{R}\times\mathbb{R}\to\mathbb{R}$ be continuous and assume that there exist constants $x_0,\delta,c\ge 0$ such that
	\begin{equation} \begin{array}{rll}  &f(x,y)\ge\delta &\mbox{ for all }x,y\in [x_0,\infty),\mbox{ and}\\
		&f(x,y)\ge -c&\mbox{ for all }x,y\in\mathbb{R}.
	\end{array}	\end{equation}
	Let $t_0\ge 0$ and suppose $d,u\colon [t_0,\infty)\to\mathbb{R}$ are continuous functions with $d(t)\le \delta$ for all $t\ge t_0$. If $z\colon [t_0-1,\infty)\to\mathbb{R}$ is continuous with $z(t_0)\le x_0$ and satisfies  
	\begin{equation}z(t)=z(t_0)+\int_{t_0}^t\left[-f(z(s-1),z(s))+d(s)\right]\mathrm ds +u(t)-u(t_0),\quad t\ge t_0,\end{equation}
	then for every $t\ge t_0$ there exists $a^t\in [t_0,t]$ such that 
	\begin{equation}z(t)\le \max\{x_0, x_0+c+\delta+u(t)-u(a^t)\}.\label{eq:uppboundz0}\end{equation}
\end{lemma}
\begin{proof}
	Fix $t\in [t_0,\infty)$. We have to deal with the case $z(t)>x_0$. Let
	\begin{equation}a^t=\sup\{s\in [t_0,t)\colon\, z(s)\le x_0\}.\end{equation}
	By continuity of $z$ we have $z(a^t)\le x_0$. If $t-a^t<1$, then
	\begin{align*}
		z(t)&= z(a^t) - \int_{a^t}^t f(z(s-1),z(s))\,\mathrm ds + \int_{a^t}^t d(s)\,\mathrm ds +u(t)-u(a^t)\\
		&\le x_0 +c +\delta +u(t)-u(a^t).\yesnumber
	\end{align*}
	Otherwise, we have $t-a^t\ge 1$. Then $z(s)\ge x_0$ for all $s\in [a^t,t]$, so $f(z(s-1), z(s))\ge \delta\ge d(s)$ for all $s\in [a^t+1,t]$, hence 
	\begin{align*}z(t)&=z(a^t)-\int_{a^t}^{a^t+1}f(z(s-1),z(s))\,\mathrm ds+\int_{a^t}^{a^t+1} d(s)\,\mathrm ds \\
		& \qquad -\int_{a^t+1}^t \left[f(z(s-1),z(s))- d(s)\right]\mathrm ds +u(t)-u(a^t)\\
		&\le x_0+c+\delta + u(t)-u(a^t),\yesnumber
	\end{align*}
    which shows the upper bound \eqref{eq:uppboundz0}.
\end{proof}

The main use of Lemma \ref{lem.zbddabove} is in Lemma \ref{lem:EsMain1}, where we show that solutions of \eqref{eq:genlized-transformed-stoch-Wright} are bounded below in probability. It also yields boundedness above in probability, as we show next. Note that there are several alternative ways of showing that solutions are bounded above in probability; see, e.g., \cite[Prop. 3.5]{vandenbosch2026b} which uses an idea of \cite{wang-wang-chen2019}. We do not need a lower bound on $a(t)$ here.
% \todo{Oorsponkelijk werd u nu gedefinieerd zonder een negative drift... Bewijs was dus incorrect, maar simpele fix. Nu correct zo? Ik gebruik alleen $a(t)\leq \alpha_1.$}

\begin{lemma}\label{thm:upper-Wright-type}
Consider the equation \eqref{eq:genlized-transformed-stoch-Wright}, where $ a(t)\le\alpha_1$ and $b(t)^2\le \si^2$ for all $t\ge 0$, for some positive constants   $\alpha_1$  and $\si$. Then every solution $z$ is bounded above in probability.
\end{lemma}

\begin{proof}
	Define \begin{equation}U(t)=-2\alpha_1t+\int_0^t a(s)\, \mathrm ds+\int_0^t b(s)\, \mathrm dW(s)\\
    =-\int_0^t [2\alpha_1-a(s)]\, \mathrm ds+\int_0^t b(s)\, \mathrm dW(s),  \end{equation}  for $t\ge 0.$ Since $\lim_{x\to \infty} G(x)=\infty$, there exists $x_0\in\mathbb{R}$ such that $G(x)\ge r+2\alpha_1$ for all $x\ge x_0$. Due to Lemma \ref{lem.zbddabove} with $f(x,y)=G(x)-r$, $\delta=2\alpha_1$, and $c=r$, we have 
	\begin{equation}z(t)\le \min\left\{x_0,x_0+r+2\alpha_1+\sup_{0\le\theta\le t} (U(t)-U(\theta))\right\},\end{equation}
	for every $t\ge 0$. By Corollary \ref{col:EsItô3abis}, we find that the  process $(\sup_{0\le \theta\le t}(U(t)-U(\theta)))_{t\ge 0}$ is bounded above in probability and, therefore, $z=(z(t))_{t\ge 0}$ is bounded above in probability. 
\end{proof}

Once a solution $z$ of \eqref{eq:genlized-transformed-stoch-Wright} is bounded in probability at integer times, the next lemma says it is bounded in probability at all times. Theorem \ref{thm:fromboundedtoinvarmeas} then yields that the process $(\|z_t\|)_{t\ge 0}$ is  bounded in probability as well. Since our proof of the lemma shows that conclusion without extra effort, we will not invoke Theorem \ref{thm:fromboundedtoinvarmeas} at this point.

% \todo{Hier stond drie keer $(\|z_t\|)_{t\ge 1}$; ik heb 1 veranderd naar 0.}

\begin{lemma}\label{lem:EsMainSegment}
	Let $G:\mathbb R\to\mathbb R$ be a continuous function satisfying \eqref{as:G1} and \eqref{as:G2}. Consider the stochastic delay equation
	\eqref{eq:genlized-transformed-stoch-Wright}. 
	Assume that there are positive constants $\alpha_1$ and $\si$ such that
	\begin{equation}0 \le a(t) \le \alpha_1\quad\mbox{and}\quad  b(t)^2 \le \si^2\qquad\hbox{for all } t \ge 0.\end{equation}
	If $z$ is a solution of \eqref{eq:genlized-transformed-stoch-Wright} such that $(z(m))_{m\in\mathbb{N}}$ is bounded in probability, then $(\|z_t\|)_{t\ge 0}$ is bounded in probability.
\end{lemma}
\begin{proof}
	Let $\varepsilon > 0$. Take $R > 0$ such that
	\begin{equation}\mathbb P\gb{|z(n)| \le R} \ge 1 - \frac{\varepsilon}{3},\end{equation}
	for every $n \in \mathbb N$. For each $n \in \mathbb N$ and $t \in [n-1,n+1]$, we have
	\begin{equation}z(n+1) = z(t) - \int_t^{n+1} G\gb{z(s-1)}\,\mathrm ds + \int_t^{n+1} a(s)\,\mathrm ds + \int_t^{n+1} b(s)\,\mathrm dW(s).\end{equation}
	This implies
	\begin{align*}
		z(t) &= z(n+1) + \int_t^{n+1} G\gb{z(s-1)}\,\mathrm ds - \int_t^{n+1} a(s)\,\mathrm ds - \int_t^{n+1} b(s)\,\mathrm dW(s).\\
		&\ge z(n+1) - 2\alpha_1 - \int_t^{n+1} b(s)\,\mathrm dW(s)\\
		&= z(n) - 2\alpha_1 - \int_{n-1}^{n+1} b(s)\,\mathrm dW(s) + \int_{n-1}^t b(s)\,\mathrm dW(s)\\
		&\ge z(n+1) - 2\alpha_1 - 2 \sup_{\theta \in [n-1,n+1]} \left|{\int_{n-1}^\theta b(s)\,\mathrm dW(s)}\right|.\yesnumber
	\end{align*}
Similarly, we have
\begin{align*}z(t) &= z(n-1) - \int_{n-1}^{t} G\gb{z(s-1)}\,\mathrm ds + \int_{n-1}^{t} a(s)\,\mathrm ds + \int_{n-1}^{t} b(s)\,\mathrm dW(s)\\
	&\le z(n-1)+2\alpha_1+\sup_{\theta\in [n-1,n+1]} \left| \int_{n-1}^\theta b(s)\, \mathrm dW(s) \right|,\yesnumber
\end{align*}
so that
\begin{equation}\sup_{t\in [n-1,n+1]}|z(t)| \le \max\{|z(n-1)|,|z(n+1)|\}+2\alpha_1+2 \sup_{\theta \in [n-1,n+1] } \left| \int_{n-1}^\theta b(s)\, \mathrm dW(s) \right|.\end{equation}
Since
	\begin{equation}\left(\int_{n-1}^{n-1+t} b(s)\,\mathrm dW(s)\right)_{t \ge 0}\end{equation}
	is a martingale, Doob's maximal inequality in Lemma \ref{lem:stochineq} yields that
	\begin{align*}
		\mathbb E\left(\sup_{t \in [n-1,n+1]} \left|{\int_{n-1}^t b(s)\,\mathrm dW(s)}\right|\right)^2 &= 
		\mathbb E\left(\sup_{t \in [0,2]} \left|{\int_{n-1}^{n-1+t} b(s)\,\mathrm dW(s)}\right|\right)^2\\
		&\le 4 \mathbb E\left(\int_{n-1}^{n+1} b(s)\,\mathrm dW(s)\right)^2 = 4 \int_{n-1}^{n+1} \mathbb E \, b(s)^2\,\mathrm ds\\
		&\le 8 \si^2.\yesnumber
	\end{align*}
	This shows
	\begin{equation}\sup_{n \ge 1} \mathbb E\left(\sup_{t \in [n-1,n+1]} \left|{\int_{n-1}^t b(s)\,\mathrm dW(s)}\right|\right)^2 \le 8\si^2,\end{equation}
	and hence, by the Markov inequality, we obtain that
	\begin{equation}\left( \sup_{t \in [n-1,n+1]}  \left|{\int_{n-1}^t b(s)\,\mathrm dW(s)}\right| \right)_{n \ge 1}\quad\mbox{is bounded in probability}.\end{equation}
	Take $S > 0$ such that
	\begin{equation}\mathbb P\left(\sup_{t \in [n-1,n+1]}  \left|{\int_{n-1}^t b(s)\,\mathrm dW(s)}\right| \ge S\right) < \frac{\varepsilon}{3}\quad\mbox{for every } n \ge 1.\end{equation}
	Then for every $n \in \mathbb N$ we obtain
	\begin{align*}
		&\mathbb P\left( \sup_{t\in [n-1,n+1]}|z(t)| \le R+2\alpha_1+2S \right)\\
		&\hspace{3cm}\ge \mathbb P\left( |z(n-1)| \le R,\ |z(n+1)|\le R, \sup_{t \in [n-1,n+1]} \int_{n-1}^t b(s)\,\mathrm dW(s) \le S\right)\\
		&\hspace{3cm}\ge 1 - \varepsilon.\yesnumber
	\end{align*}
	Thus, $(\sup_{t \in [n-1,n+1]} |z(t)|)_{n \ge 1}$ is bounded in probability.
	
	Finally, for every $t \ge 0$ there is a unique $n(t) \in \mathbb N$ such that $t \in [n(t),n(t)+1)$ and then
	\begin{equation}\sup_{\theta \in [-1,0]} |z(t+\theta)| \le \sup_{s \in [n(t)-1,n(t)+1]} |z(s)|.\end{equation}
	This shows that $(\|z_t\|)_{t \ge 0}$ is bounded in probability as well.
\end{proof}

The most involved part of our entire analysis is showing that solutions of \eqref{eq:genlized-transformed-stoch-Wright} are bounded below in probability. The next lemma presents the core of the proof. 
% \todo{Lemma 4.4 heeft update gekregen: er wordt niet meer aangenomen dat het bounded below on $[-1,0]$ is.)}\todo{Je kunt ook gewoon $\sigma_0$ weghalen?!?!  Net als hierboven. Of aannemen dat $\sigma_0$ is non-negative. Verder hoeven we hiervoor   alleen Lemma's te 3.8, 3.9 en 3.10 wijzigen (en hun bewijzen). }

\begin{lemma}\label{lem:EsMain1}
Let $G:\mathbb R\to\mathbb R$ be a continuous function satisfying \eqref{as:G1} and \eqref{as:G2}. Consider the stochastic delay equation
\eqref{eq:genlized-transformed-stoch-Wright}. 
Assume that there are positive constants  $\alpha_0$, $\alpha_1$,   and $\si$ such that
\begin{equation}\alpha_0 \le a(t) \le \alpha_1\quad\mbox{and}\quad  b(t)^2 \le \si^2\qquad\hbox{for all } t \ge 0.\end{equation}
For every solution $z=(z(t))_{t\geq -1}$, the process
\begin{equation}\label{eq:EsItô11a}
\gb{z(m)}_{m \in \mathbb N}\quad\mbox{is bounded below in probability.}
\end{equation}
\end{lemma}

\begin{proof}
First assume there is a constant $L>0$  such that $|z(\theta)|\le L$ a.s.\ for all $\theta\in [-1,0]$. 
Choose $\eta\in (0,\alpha_0)$ and choose $R>0$ so large that
\begin{equation}-L\ge -R\quad\mbox{and}\quad G(x)<\tfrac{\eta}{2}\mbox{ for all }x\in (-\infty,-R).\end{equation}
 % \todo{ik zie wel in dat het best een gedoetje is als $\alpha_0$ is 0. Dus misschien wel onnodig veel werk.} 
Choose $\beta> 32\lambda\si^2$ and  put
\begin{equation}A_0=\max\left\{x_0, x_0+\alpha_1+\tfrac{1}{2}\si^2+\beta,L\right\},\end{equation}
where $x_0$ is such that $G(x) > \alpha_1+\tfrac{1}{2}\si^2+\beta$ for all $x> x_0$. Define for $t \ge 0$ the processes
\begin{align}
U(t) &= - \beta t - \int_0^t \tfrac{1}{2}b(s)^2\,\mathrm ds + \int_0^t b(s)\, \mathrm dW(s),\label{eq:EsItô12}\\
V(t) &= \int_0^t \gb{a(s) - \eta}\,\mathrm ds + \int_0^t b(s)\,\mathrm dW(s).\label{eq:EsItô13}
\end{align}
Note that using the processes $U$ and $V$ we can rewrite equation \eqref{eq:genlized-transformed-stoch-Wright} for $z$ as
\begin{align}
\mathrm dz(t) &= -G(z(t-1))\,\mathrm dt + \eta\,\mathrm dt + \mathrm dV(t),\qquad t \ge 0,\label{eq:EsItô14}\\[0.2cm]
\mathrm dz(t) &= -\gb{G(z(t-1)) - a(t) - \tfrac{1}{2}b(t)^2 - \beta}\mathrm dt + \mathrm dU(t),\qquad t\geq 0.\label{eq:EsItô15}
\end{align}
The aim is to apply Lemma \ref{lem:EsItô3}, Lemma \ref{lem:EsItô4} and Lemma \ref{lem:EsItô5} to suitable processes $Y$.

For the process $U$, we have
\begin{equation}
U(t) = -\int_0^t \gb{\beta + \tfrac{1}{2}b(s)^2}\,\mathrm ds + \int_0^t b(s)\,\mathrm dW(s)= - \int_0^t a_U(s)\,\mathrm ds + \int_0^t b(s)\,\mathrm dW(s),
\end{equation}
where
\begin{equation}a_U(s) = \beta + \tfrac{1}{2}b(s)^2.\end{equation}
So, from the assumptions,
\begin{equation}a_U(s) \ge \beta ,\quad \sigma_0^2 \le b(s)^2 \le \si^2,\quad\mbox{and } \beta > 32\lambda \si^2.\end{equation}
This shows that the process $Y := U$ satisfies the conditions of Lemma  \ref{lem:EsItô3}.
Similarly for the process $V$, we have 
% \todo{Voor $V$ heb je dus te maken met POSITIEVE DRIFT; correct?}
\begin{equation}
V(t) = -\int_0^t \gb{\eta - a(s)}\,\mathrm ds + \int_0^t b(s)\,\mathrm dW(s)\\
= - \int_0^t a_V(s)\,\mathrm ds + \int_0^t b(s)\,\mathrm dW(s),
\end{equation}
where
\begin{equation}a_V(s) = \eta - a(s).\end{equation}
Since
\begin{equation}-a_V(s) = a(s) - \eta \le \alpha_1 - \eta\quad\mbox{and}\quad a_V(s) \le \eta - \alpha_0 < 0,\end{equation}
the process $Y := V$ satisfies the conditions of Lemma  \ref{lem:EsItô4} and Lemma \ref{lem:EsItô5}.  Note that, in contrast to $U$, this process  is an It\^o process with positive drift.

Let $\varepsilon > 0$. It follows from Lemma \ref{lem:EsItô3}, Lemma \ref{lem:EsItô4} and Lemma \ref{lem:EsItô5} that we can choose $C > 0$ so large such that, for every $m \in \mathbb N$, the processes $U$ and $V$ satisfies the following estimates
% \todo{Eerste 3 bounds ook?}
\begin{align}
&\mathbb P\left(\,\sup_{t \in [m,m+1]} \gb{V(t) - V(m)} \ge C - R - \gamma_0 - \eta \right) < \frac{\varepsilon}{6},\\
&\mathbb P\left(\,\sup_{t \in [0,1]} (V(t)-V(0)) \ge D(m-1,C)\right) <  \frac{\varepsilon}{6},\\
&\mathbb P\left(\,\sup_{t \in [0,1]} - \gb{V(m) - V(t)} \ge \tfrac{1}{4}(C-R-\gamma_0) + \tfrac{\gamma}{4}(m-1) \right) <  \frac{\varepsilon}{6},
\end{align}
and
\begin{align}
&\sum_{l=0}^{m-2} \mathbb P\left(\, \sup_{t \in [0,m-2-l]} \gb{U(m-2-l) - U(t)} \ge D(l,C) \right) <  \frac{\varepsilon}{6},\\
&\sum_{l=0}^{m-1} \mathbb P\left(\, \sup_{t \in [m-2-l,m-l]} \gb{V(t) - V(m-2-l)} \ge D(l,C) \right) <  \frac{\varepsilon}{6},\\
&\sum_{l=0}^{m-1} \mathbb P\left(\, \sup_{t \in [m-1-l,m-l]} -\gb{V(m) - V(t)} \ge \tfrac{1}{2}(C-R-\gamma_0) + \tfrac{\gamma}{4}l \right) <  \frac{\varepsilon}{6},
\end{align}
where
\begin{equation}\label{def:D}
D(l,C) :=  -\tfrac{1}{2}A_0 - \eta - \tfrac{1}{2\lambda}\log 2 +\tfrac{1}{2\lambda} \log\gb{\tfrac{1}{\gamma}(C-R-\gamma_0) + \tfrac{\eta}{2\gamma}l}.
\end{equation}

\paragraph{Goal.} Fix $m \in \mathbb N$ and define
\begin{align}
\Omega_0 &= \{ \omega \in \Omega \mid z(m)(\omega) \le -C \},\\
\Omega_1 &= \{ \omega \in \Omega_0 \mid \exists\, t \in [m,m+1] \hbox{ such that } z(t)(\omega) \ge -R \},\\
\Omega_2 &= \{ \omega \in \Omega_0 \mid z(t)(\omega) \le -R \hbox{ for all } t \in [m,m+1] \}.
\end{align}
We wish to show that $\mathbb P (\Omega_0) < \varepsilon$. Note that $\Omega_0 = \Omega_1 \cup \Omega_2$ and hence
\begin{equation}\mathbb P(\Omega_0) \le \mathbb P(\Omega_1) + \mathbb P(\Omega_2).\label{eq:POmega0}\end{equation}
So it suffices to provide estimates for $\mathbb P(\Omega_j)$, $j=1,2$.

% \smallskip
\paragraph{Estimate on $\Omega_1$.} Pointwise on $\Omega_1$, let $t^* \in [m,m+1]$ such that $z(t^*) \ge -R$. Then it follows from  \eqref{eq:EsItô14} that
\begin{equation}z(t^*) = z(m) - \int_m^{t^*} G\gb{z(s-1)}\,\mathrm ds + \eta(t^* - m) + V(t^*) - V(m),\end{equation}
so
\begin{align*}
V(t^*) - V(m) &= z(t^*) - z(m)  -  \eta(t^* - m) + \int_m^{t^*} G\gb{z(s-1)}\,\mathrm ds\\
&\ge -R + C - \gamma_0 - \eta,\yesnumber
\end{align*}
hence
\begin{equation}\sup_{t \in [m,m+1]} \gb{V(t) - V(m)} \ge -R + C - \gamma_0 - \eta,\end{equation}
so that
\begin{equation}\label{eq:EsItô16}
\mathbb P(\Omega_1) \le \mathbb P\,\left(\,\sup_{t \in [m,m+1]} \gb{V(t) - V(m)} \ge C - R - \gamma_0 - \eta \right) < \frac{\varepsilon}{6}.
\end{equation}

\paragraph{Estimate on $\Omega_2$.} Next we estimate $\mathbb P(\Omega_2)$ and we will prove that
\begin{equation}\label{eq:EsItô16a}
\mathbb P(\Omega_2) \le  \frac{5\varepsilon}{6}.
\end{equation}
Define for $l \in \{0,\ldots,m-1\}$, the subsets $\Omega_{2,l}$ of $\Omega_2$ by
\begin{equation}\Omega_{2,l} = \{ \omega \in \Omega_2 \mid z(t) \le -R \hbox{ on } [l+1,m+1), \ \exists\, t \in [l,l+1]: z(t)(\omega) > -R \}.\end{equation}
By assumption, $z(0) \ge -L \ge -R$. So $\Omega_2 = \bigcup_{l=0}^{m-1} \Omega_{2,l}$
and hence
\begin{equation}\label{eq:ESItô16a}
\mathbb P(\Omega_2) \le \sum_{l=0}^{m-1} \mathbb P(\Omega_{2,l})
\end{equation}
and in order to estimate $\mathbb P(\Omega_2)$ it suffices to estimate $\mathbb P(\Omega_{2,l})$ for $1 \le l \le m-1$.

% \smallskip
\paragraph{Estimate on $\Omega_{2,l}$.}
Pointwise on $\Omega_{2,l}$, define for $l \in \{0,\ldots,m-1\}$,
\begin{equation}t_l^* = \inf \left\{t \in [l,l+1] \mid z(s) \le -R \hbox{ for all } s \in [t,m+1] \right\}\end{equation}
and note that
\begin{equation}t_l^* \in [l,l+1]\quad\hbox{and}\quad z(t_l^*) \ge -R\quad\hbox{on } \Omega_{2,l}.\end{equation}
Fix $l \in \{0,\ldots,m-1\}$ for a moment. Pointwise on $\Omega_{2,l}$ we have
\begin{align}\label{eq:EsItô17}
z(m) &= z(t_l^* + 1) - \int_{t_l^*+1}^m G\gb{z(s-1)}\,\mathrm ds + \eta(m-t_l^*-1)\nonumber\\
&\qquad\qquad\qquad\qquad + V(m) - V(t_l^*+1).
\end{align}
So
\begin{align}\label{eq:EsItô18}
z(t_l^* + 1) &= z(m) + \int_{t_l^*+1}^m G\gb{z(s-1)}\,\mathrm ds - \eta(m-t_l^*-1)\nonumber\\
&\qquad\qquad - \gb{V(m) - V(t_l^*+1)}\nonumber\\
&\le - C + \tfrac{\eta}{2}(m-t_l^*-1) -  \eta(m-t_l^*-1)\nonumber\\
&\qquad\qquad - \gb{V(m) - V(t_l^*+1)}\nonumber\\
&\le -C - \tfrac{\eta}{2}(m-l-1) - \gb{V(m) - V(t_l^*+1)}.
\end{align}
From \eqref{eq:EsItô14} on $[t_l^*,t_l^*+1]$, we obtain
\begin{align*}
\int_{t_l^*}^{t_l^*+1} G\gb{z(s-1)}\,\mathrm ds &= z(t_l^*) + \eta + V(t_l^*+1) - V(t_l^*) - z(t_l^*+1)\\
&\ge -R + V(t_l^*+1) - V(t_l^*) - z(t_l^*+1)\\
&\ge -R + V(t_l^*+1) - V(t_l^*) + C + \tfrac{\eta}{2}(m-l-1)\\
&\qquad\qquad  + \gb{V(m) - V(t_l^*+1)},\yesnumber
\end{align*}
where in the last inequality we have used \eqref{eq:EsItô18}.

Hence there is an $s \in [t_l^*,t_l^*+1]$ such that
\begin{equation}\gamma_0 + \gamma e^{\lambda z(s-1)} \ge G\gb{z(s-1)} \ge  C - R + \frac{\eta}{2}(m-l-1) +V(m) - V(t_l^*),\end{equation}
so
\begin{equation}
\exp\,\left(\sup_{s \in [t_l^*-1,t_l^*]} \lambda z(s)\right) \ge \frac{1}{\gamma}(C-R-\gamma_0) + \frac{\eta}{2\gamma}(m-l-1) + \frac{1}{\gamma}\gb{V(m) - V(t_l^*)}.
\end{equation}
Therefore
\begin{equation}\label{eq:EsItô19}
\sup_{s \in [l-1,l+1]} \lambda z(s) \ge \log \left(\tfrac{1}{\gamma}(C-R-\gamma_0) + \tfrac{\eta}{2\gamma}(m-l-1) + \tfrac{1}{\gamma}\gb{V(m) - V(t_l^*)}\right).
\end{equation}

\paragraph{Case $l\neq 0$.} First consider the scenario $l \ge 1$. Due to Lemma \ref{lem.zbddabove}, 
with $f(x,y) = G(x)$, $d(t) = a(t) + \tfrac{1}{2}b(t)^2+\beta$, and $\delta=\alpha_1+\tfrac{1}{2}\si^2+\beta$ we have that the solution $z$ is bounded from above by
\begin{equation}\label{eq:EsItô20}
z(l-1) \le A_0 + \sup_{0 \le \theta \le l-1} \gb{U(l-1) - U(\theta)}.
\end{equation}
For $t \in [l-1,l+1]$ we obtain from \eqref{eq:EsItô14} that
\begin{align}
\nonumber z(t) &= z(l-1) - \int_{l-1}^t G\gb{z(s-1)}\,\mathrm ds + \eta(t-l+1) + V(t) - V(l-1)\\
\nonumber&\le z(l-1) + 2\eta + \gb{V(t) - V(l-1)}\\
&\le A_0 + \sup_{0 \le \theta \le l-1} \gb{U(l-1) - U(\theta)} + 2\eta +  \gb{V(t) - V(l-1)},\yesnumber
\end{align}
where in the last inequality we have used \eqref{eq:EsItô20}. Taking the supremum over $t$ yields
\begin{equation}\label{eq:EsItô21}
\sup_{t \in [l-1,l+1]} z(t) \le A_0 + \sup_{0 \le \theta \le l-1} \gb{U(l-1) - U(\theta)} + 2\eta +  \sup_{l-1 \le t \le l+1}\gb{V(t) - V(l-1)}.
\end{equation}
Combining \eqref{eq:EsItô19} and \eqref{eq:EsItô21} yields the following inequality
\begin{align}
\nonumber&\exp\,\left(\lambda\gb{A_0 + 2\eta + \sup_{0 \le t \le l-1} \gb{U(l-1) - U(t)} + \sup_{l-1 \le t \le l+1} \gb{V(t) - V(l-1)}}\right)\\
&\qquad\qquad - \frac{1}{\gamma}  \gb{V(m) - V(t_l^*)} \ge \frac{1}{\gamma}(C-R-\gamma_0) + \frac{\eta}{2\gamma}(m-l-1).
\end{align}
From this inequality it follows that at least one of the following two inequalities is satisfied. So either
\begin{align}\label{eq:EsItô22}
&\exp\,\left(\lambda\gb{A_0 + 2\eta + \sup_{0 \le t \le l-1} \gb{U(l-1) - U(t)} + \sup_{l-1 \le t \le l+1} \gb{V(t) - V(l-1)}}\right)\nonumber\\
&\qquad\qquad \ge \frac{1}{2\gamma}(C-R-\gamma_0) + \frac{\eta}{4\gamma}(m-l-1)
\end{align}
or
\begin{equation}\label{eq:EsItô23}
 - \frac{1}{\gamma}  \gb{V(m) - V(t_l^*)} \ge \frac{1}{2\gamma}(C-R-\gamma_0) + \frac{\eta}{4\gamma}(m-l-1).
\end{equation}
Inequality \eqref{eq:EsItô22} implies
\begin{align}
\nonumber&\sup_{0 \le t \le l-1} \gb{U(l-1) - U(t)} + \sup_{l-1 \le t \le l+1} \gb{V(t) - V(l-1)}\\
&\qquad\qquad \ge -A_0 - 2\eta - \frac{1}{\lambda}\log 2 + \frac{1}{\lambda}\log \left(\frac{1}{\gamma}(C-R-\gamma_0) + \frac{\eta}{2\gamma}(m-l-1)\right)
\end{align}
and inequality \eqref{eq:EsItô23} implies
\begin{equation}\sup_{t \in [l,l+1]} - \gb{V(m) - V(t)} \ge  \frac{1}{2}(C-R-\gamma_0) + \frac{\eta}{4}(m-l-1).\end{equation}

Summarizing the analysis, we have that on $\Omega_{2,l}$ at least one of the following three inequalities holds (recall $D(l,C)$ is defined in \eqref{def:D})
\begin{equation}\sup_{0 \le \theta \le l-1} \gb{U(l-1) - U(t)} \ge D(m-l-1,C)\end{equation}
or
\begin{equation}\sup_{l-1 \le t \le l+1} \gb{V(t) - V(l-1)} \ge D(m-l-1,C)\end{equation}
or
\begin{equation}\sup_{t \in [l,l+1]} - \gb{V(m) - V(t)} \ge \frac{1}{2}(C-R-\gamma_0) + \frac{\eta}{4}(m-l-1).\end{equation}
Thus it follows that
\begin{align}\label{eq:EsItô25}
\mathbb P(\Omega_{2,l}) &\le \mathbb P\,\left(\sup_{0 \le \theta \le l-1} \gb{U(l-1) - U(t)} \ge D(m-l-1,C)\right)\nonumber\\
&\qquad+ \mathbb P\,\left(\sup_{l-1 \le t \le l+1} \gb{V(t) - V(l-1)} \ge D(m-l-1,C)\right)\nonumber\\
&\qquad+ \mathbb P\,\left(\sup_{t \in [l,l+1]} - \gb{V(m) - V(t)} \ge  \frac{1}{2}(C-R-\gamma_0) + \frac{\eta}{4}(m-l-1)\right).
\end{align}

\paragraph{Case $l=0.$} Similar estimates can be obtained for the case $l=0$. Indeed, for $t \in [0,1]$ integration of \eqref{eq:EsItô14} yields
\begin{align}
\nonumber z(t) &\le z(0) - \int_0^t G\gb{z(s-1)}\,\mathrm ds + \eta t + V(t) - V(0)\\
&\le A_0 + \eta + \sup_{t \in [0,1]} \gb{V(t) - V(0)}
\end{align}
and for $t \in [-1,0]$, we have $z(t) \le A_0.$ By \eqref{eq:EsItô19} it follows that
\begin{align}
\nonumber\exp\left(\lambda\gb{A_0 + \eta + \sup_{t \in [0,1]} \gb{V(t) - V(0)}\right) - \frac{1}{\gamma}\gb{V(m) - V(t_0^*)}}\\
\qquad\qquad\ge \frac{1}{\gamma}(C-R-\gamma_0) + \frac{\eta}{2\gamma}(m-1).
\end{align}
Therefore one of the following two inequalities is satisfied,
\begin{equation}\exp\left(\lambda\gb{A_0 + \eta + \sup_{t \in [0,1]} \gb{V(t) - V(0)}} \right)\ge \frac{1}{2\gamma}(C-R-\gamma_0) + \frac{\eta}{4\gamma}(m-1)\end{equation}
or
\begin{equation}-\frac{1}{\gamma}\gb{V(m) - V(t_0^*)} \ge \frac{1}{2\gamma}(C-R-\gamma_0) + \frac{\eta}{4\gamma}(m-1).\end{equation}
Thus, on $\Omega_{2,0}$ we have
\begin{align}
\nonumber\sup_{t \in [0,1]} \gb{V(t)-V(0)} &\ge -A_0 - \eta - \frac{1}{\lambda}\log 2\\
&\qquad + \frac{1}{2\lambda} \log\left(\frac{1}{\gamma}(C-R-\gamma_0) + \frac{\eta}{2\gamma}(m-1)\right)
\end{align}
or
\begin{equation} \sup_{t \in [0,1]} - \gb{V(m) - V(t)} \ge \frac{1}{2}(C-R-\gamma_0) + \frac{\eta}{4}(m-1).\end{equation}
\paragraph{Conclusions.} By combining all estimates we obtain
\begin{align*}
\mathbb P(\Omega_2) &\le \sum_{l=0}^{m-1} \mathbb P(\Omega_{2,l}) = \mathbb P(\Omega_{2,0}) + \sum_{l=0}^{m-2} \mathbb P(\Omega_{2,m-l-1})\\
&\le
\mathbb P\left(\sup_{t \in [0,1]}  \gb{V(t)-V(0)} \ge  -A_0 - \eta - \frac{1}{\lambda}\log 2\right.\\
&\qquad\qquad\qquad\qquad\qquad\qquad +  \frac{1}{2\lambda} \log\left(\frac{1}{\gamma}(C-R-\gamma_0) + \frac{\eta}{2\gamma}(m-1)\right)\Biggr)\\
&\qquad+
\mathbb P\left( \sup_{t \in [0,1]} - \gb{V(m) - V(t)} \ge \frac{1}{2}(C-R-\gamma_0) + \frac{\eta}{4}(m-1) \right)\\
&\qquad+
\sum_{l=0}^{m-2} \mathbb P\left( \sup_{t \in [0,m-2-l]} \gb{U(m-2-l) - U(t)} \ge D(l,C) \right)\\
&\qquad+
\sum_{l=0}^{m-2} \mathbb P\left( \sup_{t \in [m-2-l,m-l]} \gb{V(t) - V(m-2-l)} \ge D(l,C) \right)\\
&\qquad+
\sum_{l=0}^{m-2} \mathbb P\left( \sup_{t \in [m-1-l,m-l]} -\gb{V(m) - V(t)} \ge \frac{1}{2}(C-R-\gamma_0) + \frac{\gamma}{4}l \right)\\
&<  \frac{5\varepsilon}{6}.\yesnumber
\end{align*}
From \eqref{eq:POmega0} we conclude $\mathbb P(\Omega_0)<\varepsilon,$ for arbitrary $m\in\mathbb N,$ which shows that for all $\varepsilon>0$ there exists a $C>0$ so that $\sup_{m\in\mathbb N}\mathbb P(z(m)\leq -C)<\varepsilon.$ This proves the assertion, provided that $|z(\theta)|\le L$   for all $\theta\in [-1,0]$,  uniformly on the entire sample space $\Omega$ (up to a null set).

Finally, define  for every  integer $L'>0$ the measurable set 
% \todo{NEW}
% \todo{Waarom kun je niet hele begin conditie doen in Hfdst 6?} 
\begin{equation}
	\Omega_{L'}:=\left\{\omega\in\Omega:\|z_0\|_\infty=\sup _{\theta \in[-1,0]}|z(\theta,\omega)| \leq L'\right\}.
\end{equation}
Then $(z(m)\mathbf 1_{\Omega_{L'}})_{m\in\mathbb N}$ is bounded below in probability for any fixed integer $L'$. Subsequently, we obtain that the   family $(z(m))_{m\in\mathbb N}$ is also bounded below in probability, irrespective  of the initial data, since $\mathbb P( \Omega_{L'}^c)=\mathbb P(\|z_0\|_\infty>L')\to 0$   as $L'\to\infty$. 
\end{proof}

We are now ready to prove the following generalized version of Theorem \ref{thm:Wrights}.

\begin{theorem}\label{thm:Wrights2}
Consider equation \eqref{eq:genSTW},
\begin{equation}
\mathrm dz(t)=-G(z(t-1))\,\mathrm dt + (r-a(z_t))\,\mathrm dt + b(z_t)\,\mathrm dW(t),
\end{equation}
where $G\colon \mathbb{R}\to\mathbb{R}$ is continuous and satisfies \eqref{as:G1} and \eqref{as:G2} for some positive constants $\gamma_0$, $\gamma$, and $\lambda$, and where $a,b\colon C[-1,0]\to\mathbb{R}$ are locally Lipschitz with $0\le a(u)\le\alpha$  and $b(u)^2\le \beta^2$, for all $u\in C[-1,0]$ for some positive constants $\alpha$ and $\beta$. 

For every $\SF_0$-measurable random variable $\varphi$ taking values in $C[-1,0]$, there exists a unique global solution $z$ with $z_0=\varphi$ and the process $(z(t))_{t\ge -1}$ is bounded above in probability. If $r > \al$, then $(z(t))_{t\ge -1}$ is also bounded below in probability and the segment process $(z_t)_{t\ge 0}$ is tight. Moreover, if $r > \al$, then there exists an invariant measure for \eqref{eq:genSTW}.
\end{theorem}
\begin{proof}
Let $\varphi$ be an $\SF_0$-measurable random variable in $C[-1,0]$. According to Proposition \ref{lem:globlexistencepuredelay} there exists a unique global solution $z$ of \eqref{eq:genSTW}. Recall that
\begin{equation}
    a(t):=r-a(z_t)\quad\hbox{and}\quad b(t):=b(z_t).
\end{equation}
% Put\todo{Deze u definieren is toch helemaal niet nodig meer? Gewoon gelijk naar thm. 4.2.}
% \begin{equation}U(t) = -\int_0^t a(z_s)\,\mathrm ds + \int_0^t b(z_s)\,\mathrm dW(s),\qquad t \ge 0.\end{equation}
% Then $U$ is a stochastic process with continuous paths, $U(0)=0$, and by Corollary \ref{col:EsItô3abis} we have that
% \begin{equation}\gb{\sup_{0 \le s \le t} (U(t) - U(s)}_{t \ge 0}\quad\mbox{is bounded above in probability.}\end{equation}
Due to Lemma \ref{thm:upper-Wright-type}, we obtain that $z$ is bounded above in probability (with $\alpha_1:=r$ and $\beta:=\sigma$).

Now assume that $r > \alpha$. To show that $z$ is also bounded below in probability, 
observe that the conditions of Lemma \ref{lem:EsMain1} are satisfied (indeed, $\alpha_0:=r-\alpha>0$), so that $(z(m))_{m\in\mathbb{N}}$ is bounded below in probability. Then Lemma \ref{lem:EsMainSegment} yields that $(\| z_t \|)_{t \ge 0}$ is bounded in probability. 
Therefore, by Theorem \ref{thm:fromboundedtoinvarmeas} we obtain that $(z_t)_{t \ge 0}$ is tight in $C[-1,0]$ and  \eqref{eq:genSTW} has an invariant measure.
\end{proof}

 \begin{proof}[Proof of Theorem $\ref{thm:Wrights}$.]
We view  \eqref{eq:stoch-t-Wrights} as a special case of \eqref{eq:genSTW},
where
\begin{equation}G(x) = re^x,\quad x \in \mathbb R,\quad\mbox{satisfies } \eqref{as:G1} \mbox{ and } \eqref{as:G2},\end{equation}
and
\begin{equation} r - a(z_t) \le r\quad\mbox{and}\quad  r - a(z_t) \ge r - \alpha > 0\quad\mbox{for all } t \ge 0.\end{equation}
 Theorem \ref{thm:Wrights} is therefore a direct corollary of  Theorem \ref{thm:Wrights2}.\end{proof}

\section{Numerical results}\label{sec:numerics}
In this section, we briefly discuss  some numerical results of the transformed Wright's equation
\begin{equation}
    x'(t)=-r\big(e^{x(t-1)}-1),\label{eq:det}
\end{equation}
together with its stochastic counterpart \eqref{sdeSTW} with a constant noise coefficient, i.e.,
\begin{equation}
    \mathrm dx(t)=-r\big(e^{x(t-1)}-1)\,\mathrm dt-\tfrac12\sigma^2\,\mathrm dt+\sigma\,\mathrm dW(t).\label{eq:numeric}
\end{equation}
Recall, this equation is obtained from the logarithmic transformation $x(t)=\log(1+y(t))$, where zero is mapped to itself. The time series shown in Figures  \ref{fig:1} and \ref{fig:2} are generated using a standard Euler–Maruyama scheme \cite{buckwar2008weak} applied to the stochastic delay equation \eqref{eq:numeric}. Invariant measures in  these figures are obtained by performing the Krylov--Bogoliubov method numerically.\newpage 

\begin{figure}[!t]
    \centering
    \includegraphics[width=0.995\linewidth, trim = 0cm 0 0 0, clip]{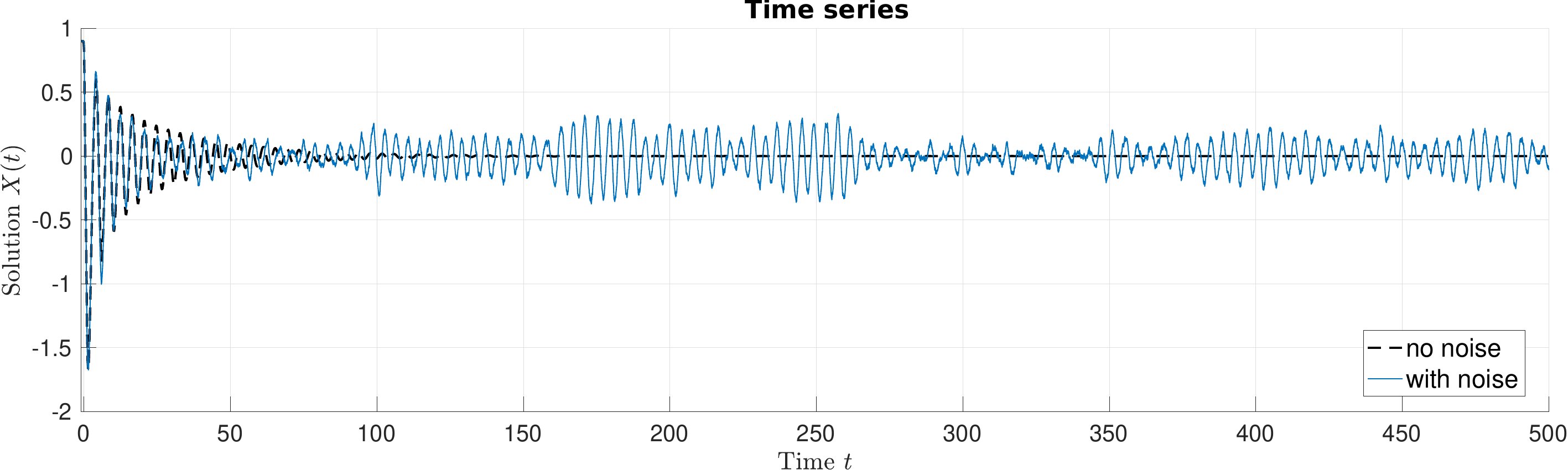}

    \,
    \includegraphics[width=0.995\linewidth, trim = 0cm 0 0 0, clip]{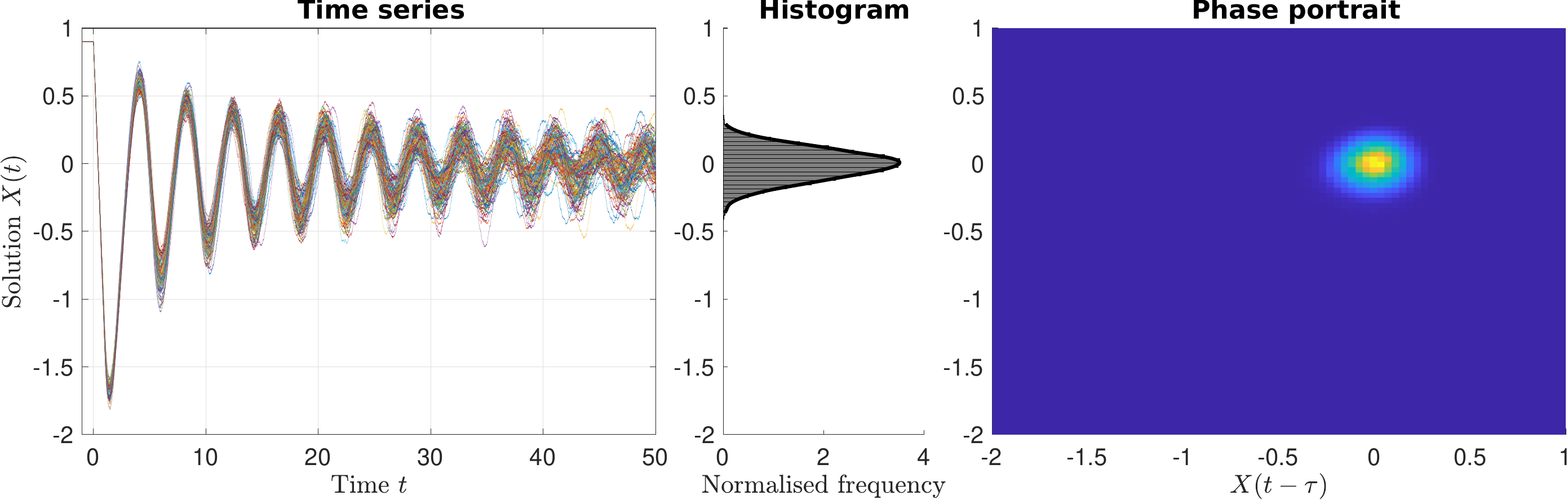}

    \caption{In the above, we plot  the  time series of both the solution to the deterministic transformed equation \eqref{eq:det} and a single realisation of the stochastic equation \eqref{eq:numeric} on the time interval $[0,500]$. The  parameters are $r=1.5$ and $\sigma=0.04$, and the initial condition is   $\varphi(t)=0.9,t\in[-1,0].$ In the below, we perform a simulation of 100 sample paths of the solution to  \eqref{eq:numeric}, with the same parameters. In the left panel, we plot these paths on the time interval $[-1,50]$. In the middle panel, we show a histogram of all values over the interval $[250,500]$, to avoid transient effects from the initial condition.  In the right panel, we display the phase portrait as a heat map,   visualising the push-forward of an invariant measure $\mu$ on $C[-1,0]$ under the evaluation map $C[-1,0]\to\mathbb R, \varphi\mapsto (\varphi(-1),\varphi(0))$. Observe that changing the interval $[250,500]$ into $[\alpha,\beta]$ for any $\beta\gg\alpha\gg 0$  leads (approximately) to the same figures, indicating that the statistical structures are indeed invariant. The invariant measure   retrieved is the stochastic analogue of the Dirac measure at 0, corresponding to the stable fixed point in \eqref{eq:det} as $r<\pi/2$.  }
    \label{fig:1}
\end{figure}

\begin{figure}[!t]
    \centering
    
    \includegraphics[width=0.995\linewidth]{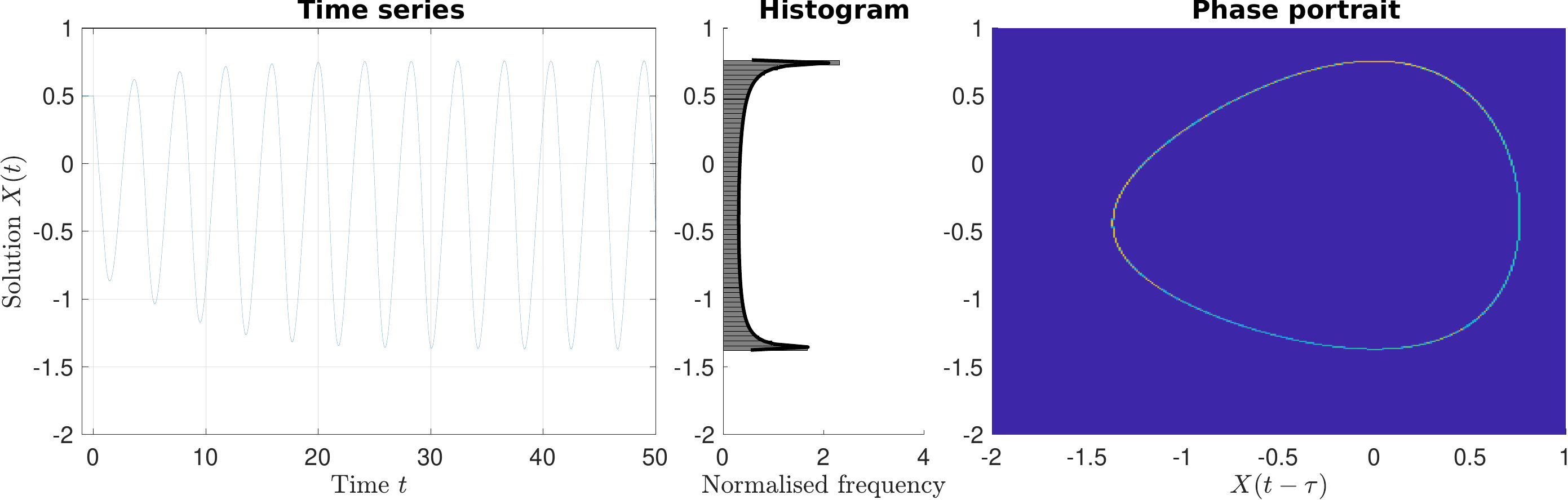}\\[.25cm]
    \includegraphics[width=0.995\linewidth]{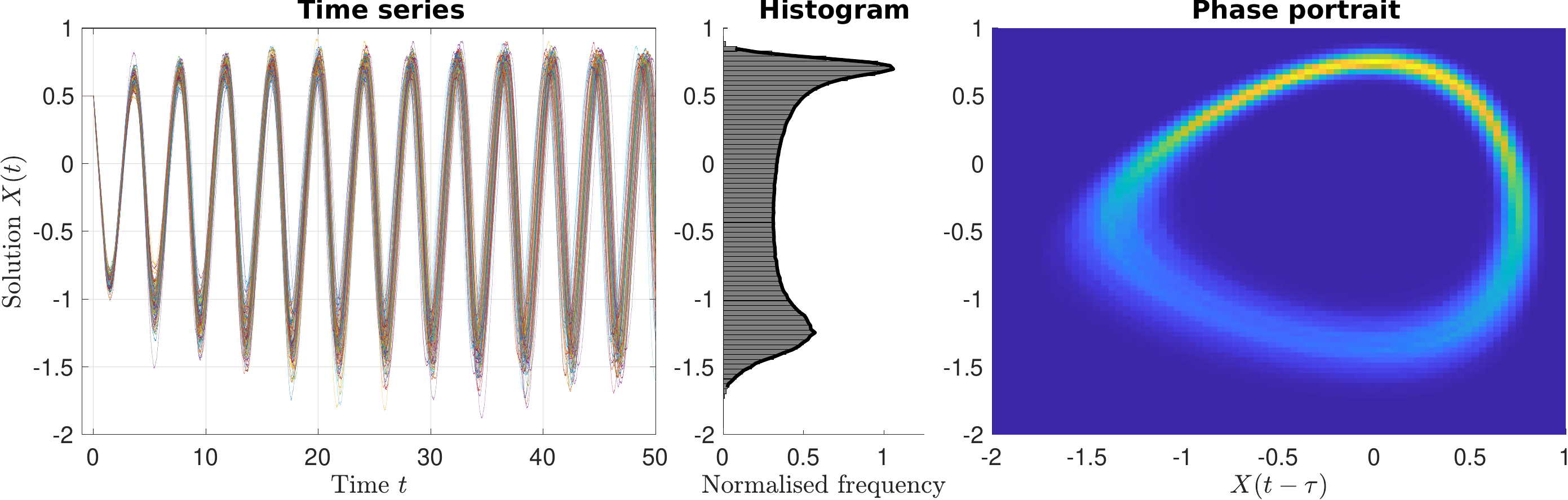}\\[.25cm]
    \includegraphics[width=0.995\linewidth]{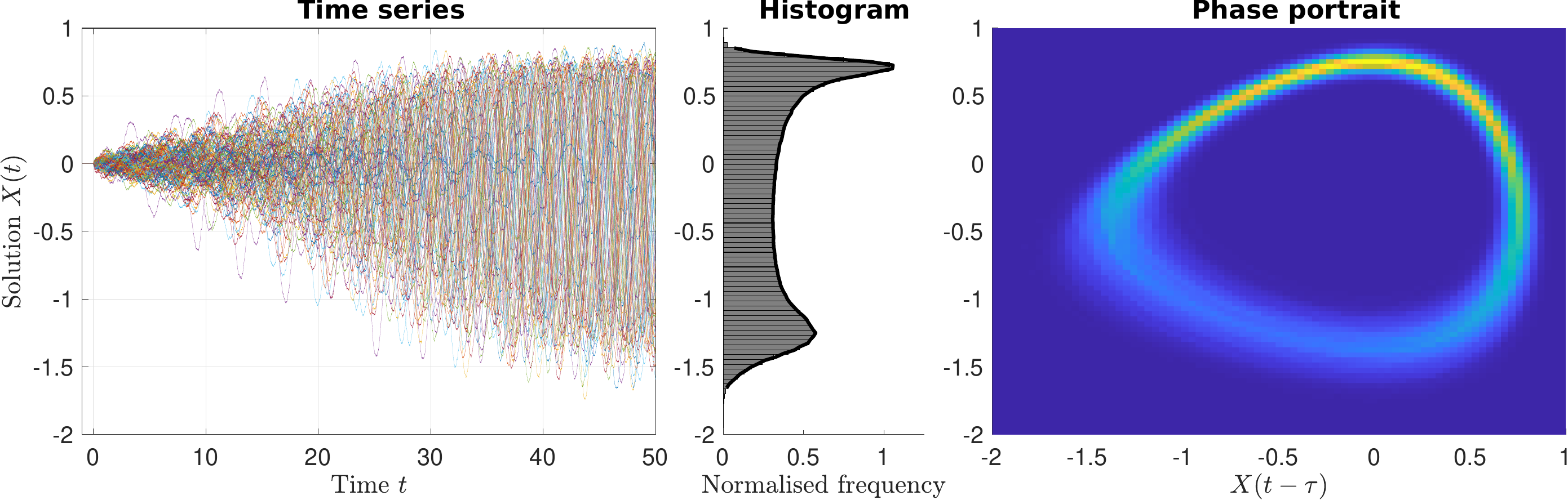}

    \caption{Simulation of the deterministic solution (first row) and that of 100 sample paths of the solution to equation \eqref{eq:numeric} (second and third row) on the time interval $[0,500]$ with parameters $r=1.75$ and $\sigma=0.04.$ The initial data is given by $z_0=\varphi,$ with $\varphi=0.5$ (first and second row) and $\varphi=0$ (third row). On the left, we plot the paths on the time interval $[-1,50]$. In the middle panel, we show a histogram of all values over the interval $[250,500]$, to avoid transient effects from the initial condition. On the right, we provide a heatmap of the phase portrait of all sample paths on the interval $[250,500]$. The invariant measures   retrieved  in the second and third row describes statistically the stochastic analogue of the unique slowly oscillating periodic orbit as $r>\pi/2$. }
    \label{fig:2}
\end{figure}

Let us explain in some detail how we obtain numerical approximations of the invariant measures present; see  \cite{vandenbosch2026b} for an extensive discussion. Define for  $0\leq t_0 \ll T<\infty $ the quantities
\begin{equation}
\nu_T(A) := \frac{1}{T} \int_{t_0}^{t_0+T} \mathbb{P}(x(t) \in A)\, \mathrm  dt,\qquad \mu_T(B) := \frac{1}{T} \int_{t_0}^{t_0+T} \mathbb{P}(x_t \in B)\, \mathrm  dt,\label{eq:muT}
\end{equation}
for measurable sets $A\subset \mathbb R$ and $B\subset C[-1,0].$ In here,  $\nu_T$ and $\mu_T$  are probability measures on $\mathbb R$ and $C[-1,0]$, respectively. 
  The Krylov--Bogoliubov theorem \cite[App. C]{vandenbosch2026a}  allows us to conclude that $(\mu_T)_{T\geq 0}$ has a convergent subsequence, whose limit we denote by $\mu$.   In many applications,  the limit   $  \mu = \lim_{T\to \infty} \mu_{T}$ simply exists (this is something we do not prove), which is hence an  invariant measure.  Note that the existence of this (subsequential) limit $\nu$   implies, for example, the existence of a stationary distribution $  \mu = \lim_{T \to \infty} \mu_T$ satisfying $\mu(A)=\nu(\{\varphi\in C[-1,0]:\varphi(0)\in A\})$. Since this construction depends on the choice of the initial condition $z_0=\varphi$, different initial data may in principle lead to different invariant measures  $\mu$ and $\nu$. 

  As a side remark, recall that properties such as support, ergodicity, or uniqueness of invariant measures obtained through the Krylov--Bogoliubov procedure are generally not known a priori. As explained in \cite{vandenbosch2026b} for the stochastic Mackey--Glass equations,  a strength of Theorem \ref{thm:Wrights} lies within the numerical implementation: it provides a rigorous framework for verifying that the  structures observed in, for instance, Figures \ref{fig:1} and \ref{fig:2} are genuine dynamical features rather than numerical artefacts. In particular, when numerical simulations reveal 
    physical finite-dimensional invariant measures, our theory indicates that these are projections of an underlying invariant measure that is  infinite-dimensional.

  In view of the above,  one obtains   a numerical approximation of a stationary distribution $\nu$ using discrete sums:
\begin{equation}
\nu(A) \approx \frac{1}{(M+1)N} \sum_{k=0}^M 
\sum_{i=1}^N \mathbf{1}_A(x_i(t_k)),\qquad \mathbf 1_A(x)=\begin{cases}
    1&x\in A,\\ 0&x\not\in A,
\end{cases}\label{CH3:eq:i:histogram}
\end{equation}
where $  x_i(t_k)$  denotes the $i$-th trajectory at time $t_k=\alpha+k\Delta t$, on a uniform grid of length $T=M\Delta t$ with step size $\Delta t$, and where $N$ denotes the number of trajectories. This discrete formulation directly corresponds to creating a histogram as in Figures \ref{fig:1} and \ref{fig:2}. The phase portraits within these figures are a numerical approximation of the push-forward measure $\nu'$ of $\mu$ under the evaluation map $\varphi\to(\varphi(-1),\varphi(0)),$ i.e., $\nu'(A')=$ $\mu(\{\varphi\in C[-1,0]:(\varphi(-1),\varphi(0))\in A'\})$ for   $A'\subset \mathbb R^2.$ Here, $\nu$ and $\nu'$ are finite-dimensional projections of the infinite-dimensional object $\mu.$

To ensure that, e.g., \eqref{CH3:eq:i:histogram} provides a good approximation of the stationary distribution, $N$ and $M$ should initially be chosen sufficiently large, and the resulting histogram should remain stable under variations in $N$ and $M$. By doing so, we tacitly exploit the fact that the limit $\nu=\lim _{T\to\infty} \nu_T$ exists (and thus not only as a subsequential limit). 
We also observe that a single trajectory ($N = 1$) suffices to approximate the invariant measure, indicating  the  measure is ergodic. 
% Indeed, the same one- and two-dimensional histograms in both Figures \ref{fig:1} and \ref{fig:2} are obtained when  a single trajectory is considered, which suggests the underlying invariant measures are ergodic.
% (i.e., $ \frac{1}{T} \int_{0}^{T} \mathbf{1}_A(X(t)) \,\mathrm  dt \to \mu(A))$).

Finally, we believe that the   invariant measures for equation \eqref{eq:numeric} are unique  whenever $\sigma\neq 0$, because  noise has a regularizing effect. 
For the deterministic  equation \eqref{eq:det},  notice that we have a unique invariant measure fully concentrated at zero for $r<\pi/2$ and multiple invariant measures for $r>\pi/2$:   one corresponding to the now unstable zero equilibrium state, one corresponding to  the unique slowly oscillating periodic orbit, and taking convex combinations of the previous two. This follows from the recently established results \cite{Berg18,jaquette2019proof}. Theoretically speaking, there may even be values  $r$ for which  rapidly oscillating periodic solutions exist \cite[Sec. 7]{jaquette2019proof}, implying the existence of another invariant measure (no matter its stability). Nevertheless, the presence of noise is expected to destroy unstable invariant structures and,  in the case of co-existing attractors, induce metastable switching between basins of attraction. As a consequence, one may indeed expect the stochastic system \eqref{sde1} to admit a unique non-trivial invariant measure.

% The existence of $\mu$ and $\nu$ in Theorem \ref{main-thrm} follows from the Krylov--Bogoliubov procedure.  In fact, we   know from this theorem  in combination with these illustrations that even much more rich underlying structures exist. In practice,  the Krylov–Bogoliubov method typically constructs the physically “relevant” invariant measures\footnote{Physically,  representing experimental time averages.  In the literature, Eckmann and Ruelle call these therefore  physical measures \cite{eckmann1985ergodic}.} for negative feedback systems; we explain why this is the case in more detail below.

%\input{Wright.bbl}
\printbibliography

\end{document}